\documentclass[reqno]{amsart}
\usepackage{amsmath, amsthm, amssymb, dsfont}
\usepackage{geometry}
\usepackage{graphicx}
\usepackage{tikz}
\usepackage{enumerate}
\usepackage{breqn}
\usepackage{soul}

\usepackage{latexsym}
\usepackage{bbm}
\usepackage{mathtools}
\usepackage{hyperref}

\usepackage{comma}
\usepackage{subcaption}
\usepackage{cutwin}
\usepackage{pstricks}
\usepackage{pst-plot}
\usepackage{pgfplots}
\usepackage{float}

\numberwithin{equation}{section}

\newcommand{\defeq}{\vcentcolon=}
\newcommand{\eqdef}{=\vcentcolon}

\newcommand{\Normal}{\mathcal{N}}

\newcommand{\N}{\mathbb{N}}

\newcommand{\R}{\mathbb{R}}

\newcommand{\1}{\mathbbm{1}}
\newcommand{\comp}{\mathsf{c}}

\newcommand{\Prob}{\mathbb{P}}
\newcommand{\Var}{\mathrm{Var}}
\newcommand{\Cov}{\mathrm{Cov}}
\newcommand{\E}{\mathbb{E}}

\newcommand{\disteq}{%
  \mathrel{\vbox{\offinterlineskip\ialign{%
    \hfil##\hfil\cr
    $\scriptscriptstyle\mathrm{law}$\cr
    \noalign{\kern.15ex}
    $=$\cr
}}}}

\newcommand{\distto}{%
  \mathrel{\vbox{\offinterlineskip\ialign{%
    \hfil##\hfil\cr
    $\scriptscriptstyle\mathrm{d}$\cr
    \noalign{\kern.15ex}
    $\to$\cr
}}}}

\newcommand{\Probto}{%
  \mathrel{\vbox{\offinterlineskip\ialign{%
    \hfil##\hfil\cr
    $\scriptscriptstyle\!\Prob$\cr
    \noalign{\kern.15ex}
    $\to$\cr
}}}}

\newcommand{\br}{\mathbf{r}}
\newcommand{\bR}{\mathbf{R}}

\usepackage{changes}

\theoremstyle{plain}
\newtheorem{theorem}{Theorem}[section]
\newtheorem{lemma}[theorem]{Lemma}
\newtheorem{proposition}[theorem]{Proposition}
\newtheorem{corollary}[theorem]{Corollary}
\theoremstyle{remark}

\theoremstyle{definition}

\newtheorem{example}[theorem]{Example}

\begin{document}
\title{A fundamental problem of hypothesis testing with finite inventory in e-commerce}
\author{Dennis Bohle \and Alexander Marynych \and Matthias Meiners}

\address{Dennis Bohle, Booking.com B.V., Amsterdam, Netherlands}
	\email{dennis.bohle@booking.com}

\address{Alexander~Marynych, Faculty of Computer Science and Cybernetics, Taras Shevchenko National University of Kyiv, Ukraine}
	\email{marynych@unicyb.kiev.ua}

\address{Matthias~Meiners, Institut f\"ur Mathematik, Universit\"at Innsbruck, Austria}
	\email{matthias.meiners@uibk.ac.at}


\date{\today}

\begin{abstract}
In this paper, we draw attention to a problem that is often overlooked or ignored
by companies practicing hypothesis testing (A/B testing) in online environments.
We show that conducting experiments on limited inventory that is shared between variants in the experiment
can lead to high false positive rates since the core assumption of independence between the groups is violated. We provide a detailed analysis of the problem in a simplified setting whose parameters are informed by realistic scenarios.
The setting we consider is a $2$-dimensional random walk in a semi-infinite strip. It is rich enough to take a finite inventory into account, but is at the same time simple enough to allow for a closed form of the false-positive probability.
We prove that high false-positive rates can occur, and develop tools that are suitable to help design adequate tests in follow-up work.
Our results also show that high false-negative rates may occur. The proofs rely on a functional limit theorem for the $2$-dimensional random walk in a semi-infinite strip.
\end{abstract}

\maketitle

\section{Introduction}
The golden standard for testing product changes on e-commerce websites is large scale hypothesis testing also known as A/B-Testing.

When a given version of a website is modified, it is natural to ask whether or not the modified
(new) version of the website performs better than the old one.
It is very common to use the following approach based on classic hypothesis testing:

\noindent
During a fixed time period, the so-called testing phase, whenever customers visit the website, they are displayed one of the two versions of it, where the choice which one they get to see is random. For each version of the website the owner thus collects a sample containing for each customer visiting that website relevant data such as whether or not they bought a good or how much money was spent by the customers, etc. Then a statistical test (A/B test) is applied to evaluate which version of the website performed better. 

Typically, these tests rely on the assumption of independent samples. In the present paper, we point out in a quantitative way that in the situation where there is a finite amount of a popular good the independence assumption is not feasible and can often lead to wrong conclusions. The inventory is shared between variants and if a copy of an item is sold it can not be bought by users that enter the experiment later. This implies that users are not independent both inside as well as between the variants. These dependencies could be avoided by randomly splitting on an item level instead of a user level, but this would reduce the choice of the customer and is therefore not a realistic setup.

We think it is best to illustrate the dependence problem with a ranking example that we will use throughout the paper. We limit ourselves to only two distinct products (which each should be thought of as a variety of different products  grouped into one). Using realistic parameters we show that two different ranking algorithms which perform identical if run independently show a significant difference almost 20\% of the time when run in an industry standard A/B experiment. We also show that if there is a difference in performance between the algorithms there are scenarios where the power of a standard A/B test is close to 0.

\begin{example}[Ranking experiment, take 1]	\label{Exa:ranking experiment}
We consider a ranking experiment with two types of goods, one rare good, which is very attractive (good $2$),
and a second, less attractive good (good $1$) available in practically unlimited quantities.
In applications, there may be more than two types of goods, but the less attractive ones are labeled as type-$1$ goods,
while the most attractive ones are labeled as type-$2$ goods.
Suppose that in total there are $1\,000$ goods of type $2$ and $1\,000\,000$ goods of type $1$.

A website displays the available goods to each visitor.
The goods are displayed in a certain order, which depends on the ranking algorithm used.
The owner of the website wants to compare two different algorithms.
The default algorithm, Algorithm $0$, displays the goods such that the type-$2$ goods have a low ranking and appear late in the list.
Thus, only a fraction of the visitors gets to see them.
The new algorithm, Algorithm $1$, displays the goods such that the type-$2$ goods have the highest ranking and appear first in the list.
Every visitor seeing the goods ranked by Algorithm $1$ will see both, type-$1$ and type-$2$ goods (as long as they are available).
The goal is to find out which of the two algorithms leads to a higher overall conversion rate, i.e., to a higher empirical probability to make a sale.

Suppose that during a test phase, $n=4\,000\,000$ customers visit the website.
Whenever a customer visits the website, a fair coin is tossed.
If the coin shows heads, the products are displayed ranked according to Algorithm $1$,
if the coin shows tails, the products are displayed ranked according to Algorithm $0$.

We now make the following model assumptions.
We assume that, independent of all other customers,
each customer has a chance of $20\%$ of preferring good $1$ over good $2$
and an $80\%$ chance of preferring good $2$ over good $1$.
When the goods are ranked according to Algorithm $0$,
the customer first sees type-$1$ goods.
There is a $5\%$ chance that the customer scrolls down and spots a type-$2$ good (if still available).
A customer who sees both goods and has a preference for good $2$ will buy good $2$ with $10\%$ chance
and will not buy at all with $90\%$ chance.
A customer who either sees both goods and has a preference for good $1$
or sees only good $1$ will buy good $1$ with a $5\%$ chance and will not make a buy at all with $95\%$ chance.
For simplicity, we assume that each customer buys at most one good.
\begin{figure}[h]
\begin{tikzpicture}[level distance=2cm,level 1/.style={sibling distance=7cm},
level 2/.style={sibling distance=4cm},
level 3/.style={sibling distance=2.5cm},
  every node/.style = {shape=rectangle, rounded corners,
    draw, align=center,
    top color=white, bottom color=black!10}]
  \node {Algorithm 0}
    child {node {sees only\\ type $1$}
		child{node {buys\\ type $1$} edge from parent node[left] {$0.05$}}
    		child{node {no buy} edge from parent node[right] {$0.95$}}
    edge from parent node[left] {$0.95$} }
        child {node {sees\\ types $1\&2$}
        	child{node {prefers\\ type $1$}
    		child{node {buys\\ type $1$} edge from parent node[left] {$0.05$}}
        		child{node {no buy} edge from parent node[right] {$0.95$}}
    	edge from parent node[left] {$0.20$}}
        	child{node {prefers\\ type $2$}
    		child{node {buys\\ type $2$} edge from parent node[left] {$0.10$}}
        		child{node {no buy} edge from parent node[right] {$0.90$}}
        edge from parent node[right] {$0.80$}}
    edge from parent node[right] {$0.05$}};
\end{tikzpicture}
	\caption{Ranking experiment: Algorithm 0}	\label{fig:Algo 0}
\end{figure}
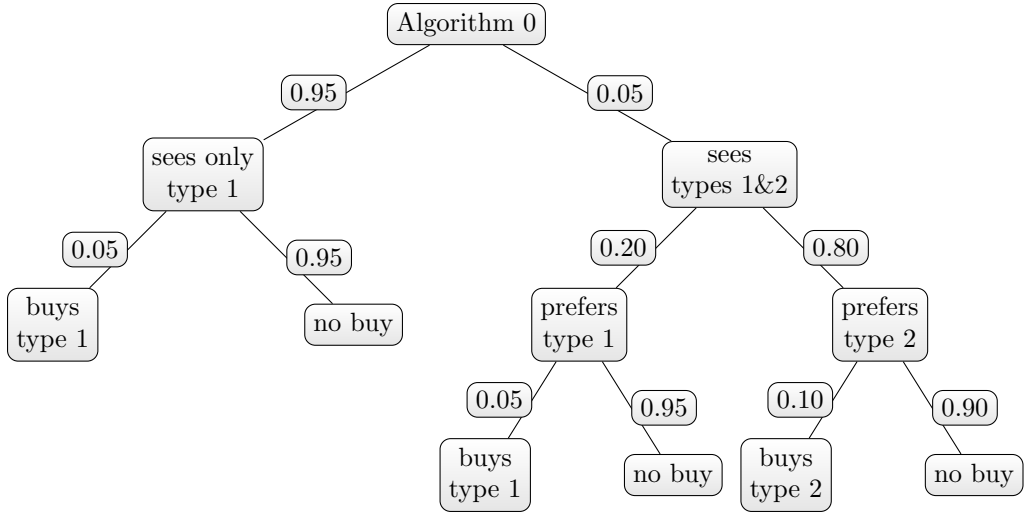
\begin{figure}[h]
\begin{tikzpicture}[level distance=1.7cm,level 1/.style={sibling distance=7cm},
level 2/.style={sibling distance=5cm},
level 3/.style={sibling distance=3.2cm},
  every node/.style = {shape=rectangle, rounded corners,
    draw, align=center,
    top color=white, bottom color=black!10}]
  \node {Algorithm 1}
        child {node {sees\\ types $1\&2$}
        	child{node {prefers\\ type $1$}
    		child{node {buys\\ type $1$} edge from parent node[left] {$0.05$}}
        		child{node {no buy} edge from parent node[right] {$0.95$}}
    	edge from parent node[left] {$0.20$}}
        	child{node {prefers\\ type $2$}
    		child{node {buys\\ type $2$} edge from parent node[left] {$0.10$}}
        		child{node {no buy} edge from parent node[right] {$0.90$}}
        edge from parent node[right] {$0.80$}}
edge from parent node[right] {$1.0$}};
\end{tikzpicture}
	\caption{Ranking experiment: Algorithm 1}	\label{fig:Algo 1}
\end{figure}
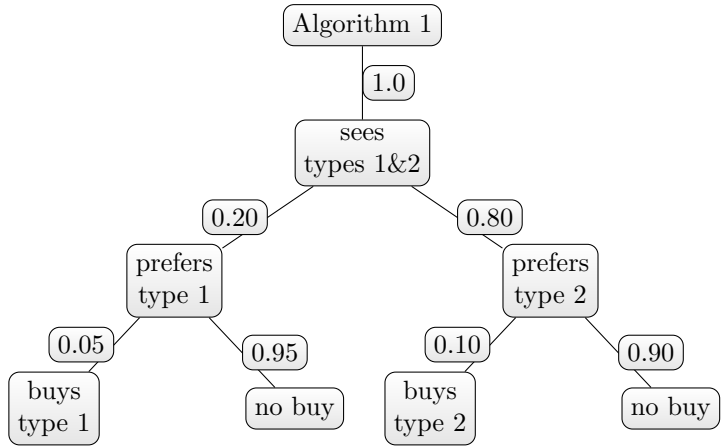

The data collected is a sample $(x_1,y_1,i_1),\ldots,(x_n,y_n,i_n)$ where $n$ is the sample size,
i.e., the number of customers visiting the website during a certain test period.
Here, $i_k$ is either $0$ or $1$, depending on whether Algorithm $0$ or $1$ was used to display the goods to the $k^{\rm th}$ customer.
Further, $x_k=1$ or $x_k=0$ depending on whether good 1 was bought or not and, analogously,
$y_k=1$ or $y_k=0$ depending on whether good 2 was bought or not.
Notice that by our assumption that each customer buys at most one good, we have $x_k+y_k \in \{0,1\}$.
Those $(x_k,y_k,i_k)$ with $i_k=0$ are assigned to sample $0$ and the others to sample $1$.
We write $n_0 \defeq \sum_{k=1}^n (1-i_k)$ and $n_1 \defeq \sum_{k=1}^n i_k$
for the corresponding sample sizes.
The numbers of sales in each group are $\ell_0 = \sum_{k=1}^n (x_k+y_k) (1-i_k)$
and $\ell_1 = \sum_{k=1}^n (x_k+y_k) i_k$,
the total number of sales is $\ell = \ell_0 + \ell_1$.
The empirical probabilities for sales in samples $0$ and $1$ are
\begin{equation*}
p_0 \defeq \frac1{n_0} \sum_{k=1}^n (x_k+y_k) (1-i_k) = \frac{\ell_0}{n_0}
\quad	\text{and}	\quad
p_1 \defeq \frac1{n_1} \sum_{k=1}^n (x_k+y_k) i_k  = \frac{\ell_1}{n_1}.
\end{equation*}
The website owner wants to find out whether Algorithm $1$ performs better than Algorithm $0$.
\begin{figure}[h]
\begin{subfigure}[t]{0.45\textwidth}
\begin{center}
\includegraphics[width=1\textwidth]{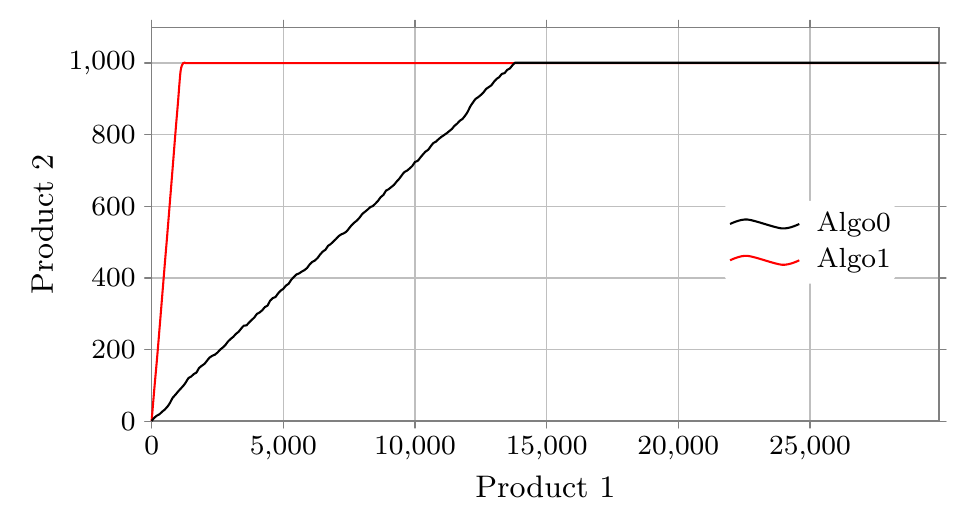}
\end{center}
\subcaption{Simulation of sales for Algorithms $0$ and $1$ run on separate inventories.
Both algorithms sell all 1\,000 attractive goods.
Algorithm $0$ additionally sells $199\,528$ goods of type $1$, Algorithm $1$ sells $199\,325$.
The differences between the two algorithms are within the fluctuations one expects.
Surely, this simulation does not give rise to the conclusion that Algorithm $1$ outperforms Algorithm $0$.
However, Algorithm $1$ sells the attractive goods earlier.}
\end{subfigure}
\hfill
\begin{subfigure}[t]{0.45\textwidth}
\begin{center}
\includegraphics[width=1\textwidth]{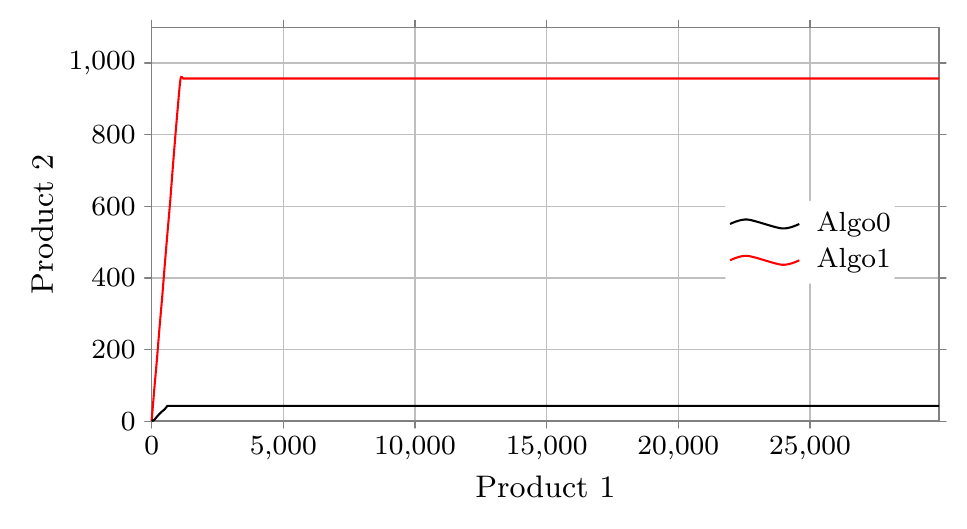}
\end{center}
\subcaption{Simulation of sales for Algorithms $0$ and $1$ run on shared inventory.
Algorithm $0$ sells $99\,224$ items of good $1$ and $43$ items of good $2$,
while Algorithm $1$ sells $99\,229$ items of product $1$ and the remaining $957$ items of product $2$.
The huge difference in sales of the attractive good
leads to a rejection of the null hypothesis (of both algorithms performing equally well) by the chi-squared test,
the $p$-value in this simulation is $0.0369\ldots < 0.05$.}
\end{subfigure}
\caption{\small Simulation of sales in Example \ref{Exa:ranking experiment}
on separate and shared inventory.}
\end{figure}

It is a common approach to test for the higher probability of a sale by assuming an independent sample and using a $G$-test
or the asymptotically equivalent two-sample chi-squared test.
The hypothesis is that the conversion rates are identical in both samples.
For simplicity, in the paper at hand, we shall always consider the chi-squared test.
The test statistics for the latter is
\begin{equation*}
\chi^2 = \sum_{i=0,1} \frac{(\ell_i-\ell\frac{n_i}{n})^2}{\ell \frac{n_i}{n}}
+ \sum_{i=0,1} \frac{(n_i\!-\!\ell_i-\ell\frac{n_i}{n})^2}{(n-\ell) \frac{n_i}{n}}.
\end{equation*}
The hypothesis is rejected if $\chi^2 > q_{1-\alpha}$ where $\alpha \in (0,1)$ is the significance level
and $q_{1-\alpha}$ is the $(1\!-\!\alpha)$-quantile of the chi-squared distribution with one degree of freedom,
see \cite[Chapter 17]{vdVaart:1998}.

Throughout the paper, we shall return repeatedly to Example \ref{Exa:ranking experiment}
and discuss it in the light of our findings.
\end{example}

We shall discuss a variant of this example later on showing that ignoring the dependencies might also lead
to too high false negative rates, see Example \ref{Exa:picky customers} below.

\section{Model assumptions}	\label{sec:model}

We return to the general situation,
in which a website offers two types of goods, good $1$ and good $2$.
During a test phase, in which a new website design is used in parallel, the website has $n$ visitors.
Suppose that the website has a practically unlimited supply of items of good $1$,
while there are only $c_n \in \{1,2,3,\ldots\} \eqdef \N$ items of good $2$.
Typically, $n$ will be very large and $c_n$ will also be large, but significantly smaller than $n$.
Whenever a user visits the website, a coin
with success probability $p$ is tossed.
If the coin shows heads, the new website design is displayed, whereas if the coin shows tails,
the old design is displayed.
We thus observe a sample $((x_1,y_1,i_1),\ldots,(x_n,y_n,i_n)) \in(\N_0^2 \times \{0,1\})^n$ where $\N_0 \defeq \N \cup \{0\}$.
Here,
$x_k$ and $y_k$ are the numbers of goods of type $1$ and $2$, respectively, that have been bought by the $k^{\rm th}$ visitor of the website during the test phase,
while $i_k = 1$ if the new design has been displayed to the $k^{\rm th}$ visitor, and $i_k = 0$, otherwise.
We consider $((x_1,y_1,i_1),\ldots,(x_n,y_n,i_n))$ as the realization of a random vector
$((X_1,Y_1,I_1),\ldots,(X_n,Y_n,I_n))$. We define $Z_k \defeq \1_{\{X_k+Y_k > 0\}}$ to be the indicator
of the event that the $k^{\rm th}$ customer bought something.
Further, we set $\br_k \defeq (s_k,t_k) \defeq (x_1,y_1)+\ldots+(x_k,y_k)$ and
$\bR_k \defeq (S_k,T_k) \defeq (X_1,Y_1)+\ldots+(X_k,Y_k)$
for $k=0,\ldots,n$ where the empty sum is defined to be the zero vector.

\subsection{The classical model assuming independence}

\begin{sloppypar}
Many website owners in e-commerce use the $G$-test or the chi-squared test in the given situation.
This test only uses the information whether or not a good was purchased,
that is, the only information from the sample
$((X_1,Y_1,I_1),\ldots,(X_n,Y_n,I_n))$
used by the test is $(Z_1,I_1),\ldots,(Z_n,I_n)$.
This amounts to the following model assumptions.
\end{sloppypar}

\begin{itemize}\itemsep1pt
	\item[($\chi1$)]
		There is a sequence $(I_1,I_2,\ldots)$ of i.i.d.\ copies of a Bernoulli variable $I$
		with $\Prob(I=1)=p=1-\Prob(I=0) \in (0,1)$.
	\item[($\chi2$)]
		There are a random variable $\zeta$ and $p_0,p_1 \in (0,1)$
		such that
		\begin{equation*}
		\Prob(Z_k \in \cdot | I_k = i) = \Prob(\zeta \in \cdot | I=i) = \mathrm{Ber}(p_i)(\cdot) = p_i \delta_1(\cdot) + (1-p_i) \delta_0(\cdot)
		\end{equation*}
		for all $k \in \N$ and $i=0,1$,
	\item[($\chi3$)]
		The family $((I_k,Z_k))_{k \in \N}$ is independent.
\end{itemize}
Here, and throughout the paper, for $x \in \R^d$, we write $\delta_x$ for the Dirac measure with a point at $x$.
Assumptions ($\chi1$) through ($\chi3$) possess the following interpretations.

\noindent
($\chi1$):
The random variable $I_k$ models the coin toss that is used to decide whether the new or the old website design is displayed to the $k^{\rm th}$ visitor of the website during the test phase.

\noindent
($\chi2$):
The random variable $Z_k$ is the indicator of the event that the $k^{\rm th}$ visitor bought something.

\noindent
($\chi3$):
The independence assumption in the context of the low inventory problem is made for simplicity.
We question the feasibility of this assumption in the present paper.

\subsection{A model incorporating low inventory of a popular good.}

We propose a simple model in which we keep track of the inventory of a rare good.
Throughout the paper, we shall refer to this model as the `model incorporating low inventory'.
By $c_n \in \N$ we denote the quantity at which the rare good is available.
The most important case we consider is where $c_n$ is asymptotically equivalent to a constant times $\sqrt{n}$.
However, as we need this assumption only occasionally,
throughout the paper, we only assume that $(c_n)_{n\in\N}$ is a non-decreasing unbounded sequence of integers
which is regularly varying with index $\rho\in(0,1]$ at infinity\footnote{see \cite{Bingham+Goldie+Teugels:1987} for a standard textbook reference}, that is,
\begin{equation}	\label{eq:c_n reg var}
\lim_{n\to\infty}\frac{c_{\lfloor nt\rfloor}}{c_n}=t^{\rho},	\quad	t \geq 0
\end{equation}
and further $c_n = O(n)$ as $n\to\infty$, which is relevant only if $\rho = 1$.
Notice that the case $c_n \sim \mathrm{const} \cdot \sqrt{n}$ is covered.
Indeed, in this case, we have $\rho=\frac12$.

In the next step, we informally describe the evolution of the process $(\bR_k)_{k \in \N_0}$.
Let $C_n \defeq \N_0 \times ([0,c_n] \cap \N_0)$,
$C_n^\circ \defeq \N_0 \times ([0,c_n) \cap \N_0)$
and $\partial C_n \defeq C_n \setminus C_n^\circ = \N \times \{c_n\}$.
At each step, a coin with success probability $p$ is tossed.
Depending on whether the coin shows heads or tails, the walk attempts to make one step
according to a probability distribution $\mu_1$ or $\mu_0$, respectively, on $\N_0^2$.
The step is actually performed if the walk stays in the strip $C_n$.
Otherwise, another independent coin with success probability $q$ is tossed.
If the second coin shows heads, the walk moves in each coordinate direction according to the attempted step as far as possible
but stops at the boundary of $C_n$. If the second coin shows tails, the walk stays put.
Once the walk is on the boundary of $C_n$, it moves there according to a one-dimensional random walk in horizontal direction.\smallskip

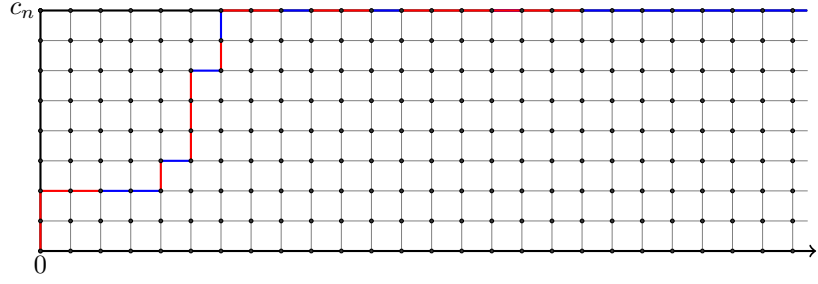
\begin{figure}[h]
\begin{center}
\begin{tikzpicture}[thin, scale=0.4,-,
                   shorten >=0pt+0.5*\pgflinewidth,
                   shorten <=0pt+0.5*\pgflinewidth,
                   every node/.style={circle,
                                      draw,
                                      fill          = black!80,
                                      inner sep     = 0pt,
                                      minimum width = 1.4 pt}]

\draw [help lines] (0,0) grid (25.5,8);
\draw [thick,->] (0,0) -- (25.8,0);
\draw [thick,-] (0,0) -- (0,8);
\draw [thick,-] (0,8) -- (25.5,8);
    \node[label=below:$0$] at (0,0) {};
    \node[label=left:$c_n$] at (0,8) {};
    
\draw [thick,-,color=red] (0,0) -- (0,1);
\draw [thick,-,color=red] (0,1) -- (0,2);
\draw [thick,-,color=red] (0,2) -- (1,2);
\draw [thick,-,color=red] (1,2) -- (2,2);
\draw [thick,-,color=blue] (2,2) -- (3,2);
\draw [thick,-,color=blue] (3,2) -- (4,2);
\draw [thick,-,color=red] (4,2) -- (4,3);
\draw [thick,-,color=blue] (4,3) -- (5,3);
\draw [thick,-,color=red] (5,3) -- (5,4);
\draw [thick,-,color=red] (5,4) -- (5,5);
\draw [thick,-,color=red] (5,5) -- (5,6);
\draw [thick,-,color=blue] (5,6) -- (6,6);
\draw [thick,-,color=red] (6,6) -- (6,7);
\draw [thick,-,color=blue] (6,7) -- (6,8);
\draw [thick,-,color=red] (6,8) -- (7,8);
\draw [thick,-,color=red] (7,8) -- (8,8);
\draw [thick,-,color=blue] (8,8) -- (9,8);
\draw [thick,-,color=red] (9,8) -- (10,8);
\draw [thick,-,color=red] (10,8) -- (11,8);
\draw [thick,-,color=blue] (11,8) -- (12,8);
\draw [thick,-,color=red] (12,8) -- (13,8);
\draw [thick,-,color=red] (13,8) -- (14,8);
\draw [thick,-,color=red] (14,8) -- (15,8);
\draw [thick,-,color=red] (15,8) -- (16,8);
\draw [thick,-,color=blue] (15,8) -- (16,8);
\draw [thick,-,color=red] (15,8) -- (16,8);
\draw [thick,-,color=red] (16,8) -- (17,8);
\draw [thick,-,color=red] (17,8) -- (18,8);
\draw [thick,-,color=blue] (18,8) -- (19,8);
\draw [thick,-,color=blue] (19,8) -- (20,8);
\draw [thick,-,color=blue] (20,8) -- (21,8);
\draw [thick,-,color=blue] (21,8) -- (22,8);
\draw [thick,-,color=blue] (22,8) -- (23,8);
\draw [thick,-,color=blue] (23,8) -- (24,8);
\draw [thick,-,color=blue] (24,8) -- (25,8);
\draw [thick,-,color=blue] (25,8) -- (25.5,8);

\foreach \x in {0,1,2,3,4,5,6,7,8,9,10,11,12,13,14,15,16,17,18,19,20,21,22,23,24,25}
\foreach \y in {0,1,2,3,4,5,6,7,8}
    \node at (\x,\y) {};

%
\end{tikzpicture}
\end{center}
\caption{The $2$-dimensional random walk in a semi-infinite strip. The $k^{\rm th}$ step is drawn in red if $I_k=1$ and it is drawn in blue, otherwise.
The walk moves in the strip until it hits the upper boundary
and continues on the boundary as a $1$-dimensional random walk.
Here, we use the particular model from Example \ref{Exa:ranking experiment}.
Notice that the majority of steps have length zero and are not displayed.}
\end{figure}
\smallskip

The underlying model assumptions are the following.
\begin{itemize}\itemsep1pt
	\item[(A1)]
		There are two sequences $(I_1,I_2,\ldots)$
		and $(J_1,J_2,\ldots)$ of i.i.d.\ copies of Bernoulli variables $I$ and $J$, respectively,
		with $\Prob(I=1)=p=1-\Prob(I=0) \in (0,1)$ and $\Prob(J=1)=q=1-\Prob(J=0) \in [0,1]$.
	\item[(A2)]
		There are a sequence
		$(\xi_1,\eta_1),(\xi_2,\eta_2),\ldots$ of i.i.d.\ copies of 		a random variable $(\xi,\eta)$
		and two probability measures $\mu_0$, $\mu_1$ on $\N_0^2$
		satisfying $\mu_i(\{(a,b)\}) > 0$
		for all $(a,b) \in \{0,1\}^2$ and $i=0,1$
		such that
		\begin{equation*}
		\Prob((\xi_k,\eta_k) \in \cdot | I_k = i)
		= \Prob((\xi,\eta) \in \cdot | I=i) = \mu_i(\cdot)
		\end{equation*}
		for all $k \in \N$ and $i=0,1$.
		We set $\mu_p(\cdot) \defeq p \mu_1(\cdot) + (1-p) \mu_0(\cdot)
		= \Prob((\xi,\eta) \in \cdot)$.
	\item[(A3)]
		There are a probability measure $\nu$ on $\N_0$
		and i.i.d.\ copies $\theta_1,\theta_2,\ldots$
		of a random variable $\theta$ with $\Prob(\theta \in \cdot) = \nu(\cdot)$.
	\item[(A4)]
		The sequences $(I_k)_{k \in \N}$, $(J_k)_{k \in \N}$, $((\xi_k,\eta_k))_{k \in \N}$
		and $(\theta_k)_{k \in \N}$ are independent; all random variables have finite second moments.
	\item[(A5)]
		Let $k \in \N$. If $\bR_{k-1} = (S_{k-1},T_{k-1}) \in C_n^\circ$, then
		\begin{equation*}
		\bR_k =	\begin{cases}
					\bR_{k-1} + (\xi_k,\eta_k)					&	\text{if }	T_{k-1}+\eta_k \leq c_n,	\\
					(S_{k},(T_{k-1} + \eta_k) \wedge c_n)	&	\text{if }	T_{k-1}+\eta_k > c_n	\text{ and } J_k = 1,	\\
					\bR_{k-1}								&	\text{if }	T_{k-1}+\eta_k > c_n	\text{ and } J_k = 0.
					\end{cases}
		\end{equation*}
		On the other hand, if $\bR_{k-1} \in \partial C_n$, then $T_{k-1} = c_n$.
		In this case,
		\begin{equation*}
		\bR_k = \bR_{k-1} + (\theta_k,0).
		\end{equation*}
		
		\noindent
		Finally, define $(X_k,Y_k) \defeq \bR_k-\bR_{k-1}$.
\end{itemize}
The interpretations of these assumptions are the following.

\noindent
(A1):
The random variable $I_k$ models the coin toss that is used to decide whether the new or the old website design is displayed to the $k^{\rm th}$ visitor of the website during the test phase.
The random variable $J_k$ models the preference of the $k^{\rm th}$ visitor.
If $J_k=1$, then the user must buy. Users with $J_k=0$ only buy when they get exactly what they want in the first place.

\noindent
(A2):
The random variable $(\xi_k,\eta_k)$ can be interpreted as the vector of goods that the $k^{\rm th}$ visitor would buy when visiting the (displayed version of the) website
if there was enough supply of these goods.

\noindent
(A3):
The random variable $\theta_k$ can be interpreted as the amount of type-$1$ goods that the $k^{\rm th}$ visitor would buy
when visiting the website and finding only goods of type $1$ left.

\noindent
(A4):
This is an independence assumption which is made to keep the model as simple as possible.

\noindent
(A5):
The random variable $(X_k,Y_k)$ models what is actually bought by the $k^{\rm th}$ user.
This depends on the needs of the user, $\xi_k$, $\eta_k$ and $\theta_k$, the remaining amount of the rare good $2$ given by $c_n-T_{k-1}$,
and the user's preference $J_k$.
If there are enough goods available to meet the needs of the $k^{\rm th}$ user,
then the user will buy exactly the needed amounts,
namely, $\xi_k$ of good $1$ and $\eta_k$ of good $2$.
If good $2$ is not available at a sufficient quantity, then the user will either buy as much as possible of each of the goods if $J_k=1$
or nothing at all if $J_k=0$. If there is nothing left of good $2$, the user will buy $\theta_k$ of good $1$.\footnote{
Notice that according to our model, the two versions of the website have an identical effect on the user
once the popular good is sold out. This is a simplifying assumption which excludes situations
where for instance the effect of a new banner on the website is investigated.
These situations can sometimes be analyzed via classical tests.
In any case, we point out that our proofs could be easily modified to deal with the situation
where the law of $\theta_k$ depends on the value of $I_k$,
but then the results become even more cumbersome.
}

Notice that in both models, $\Prob$ depends on $p$, which is not explicit in the notation.
While in large parts of the paper, $p$ is fixed, in some places, however, it is important to make the dependence
of $\Prob$ on $p$ explicit. In these places, we write $\Prob_p$. Often, this will be $\Prob_0$ or $\Prob_1$,
which correspond to the situations where only one version of the website is used.

Let us introduce some notation for various characteristics of the above variables which we shall  use throughout the paper.
\begin{itemize}
	\item
		We set ${\bf m}_0\defeq (m^{\xi}_0,m^{\eta}_0) \defeq \E[(\xi,\eta) | I=0]$,
		${\bf m}_1 \defeq (m^{\xi}_1,m^{\eta}_1) \defeq \E[(\xi,\eta) | I=1]$
		and ${\bf m} \defeq (m^{\xi},m^{\eta}) \defeq \E[(\xi,\eta)] = p{\bf m}_1+(1-p){\bf m}_0$.
		Notice that $m^{\xi}$ and $m^{\eta}$ depend on $p$ even though this is not explicit in the notation.
	\item
		The covariance matrices of the probability measures $\mu_i$, $i=0,1$
		are denoted by ${\bf C}_i$, $i=0,1$, respectively.
		The covariance matrix of the probability measure $\mu_p$ is then
		\begin{equation*}
		{\bf C} = p{\bf C}_1+(1-p){\bf C}_0
		=\left(\begin{matrix}\sigma_\xi^2 & \rho_{\xi\eta}	\\ \rho_{\xi\eta} & \sigma^2_{\eta}\end{matrix}\right).
		\end{equation*}
	\item
		We denote by $m^{\theta}=\E[\theta]$ and $\sigma^2_{\theta}= \Var[\theta]$,
		the mean and the variance of the probability measure $\nu$.
	\item
		Finally,
		we set $p_i \defeq \mu_i(\{(0,0)^\comp\}) = \Prob(\xi+\eta > 0 | I = i)$
		for $i=0,1$
		and $p_\theta \defeq \Prob(\theta>0)$.
		The $p_i$, $i=0,1$ and $p_\theta$
		are the theoretical conversion rates.
\end{itemize}

\begin{example}[Ranking experiment, take 2]	\label{Exa:ranking experiment as special case of our model}
We return to Example \ref{Exa:ranking experiment} and briefly explain how this example fits into the framework
of the above model.
The number of visitors during the test phase is $n=4 \cdot 10^6$.
The quantity of the attractive good $2$ is $c_n = 1\,000 = \frac12 \cdot \sqrt{n}$.
Website visitors view each version of the website with equal probability,
so $I_1,I_2,\ldots$ have success probability $p = \frac12$.

Further, as can be readily seen from Figure \ref{fig:Algo 0},
\begin{align*}	\textstyle
\mu_0 = \frac{96}{100} (\frac1{20} \delta_{(1,0)} + \frac{19}{20} \delta_{(0,0)})
+ \frac{4}{100} (\frac1{10} \delta_{(0,1)} + \frac{9}{10} \delta_{(0,0)})
= \frac{948}{1\,000} \delta_{(0,0)} + \frac{4}{1\,000} \delta_{(0,1)} + \frac{48}{1\,000} \delta_{(1,0)}.
\end{align*}
Analogously, from Figure \ref{fig:Algo 1}, we deduce
\begin{align*}	\textstyle
\mu_1 = \frac{1}{5} (\frac1{20} \delta_{(1,0)} + \frac{19}{20} \delta_{(0,0)})
+ \frac45 (\frac1{10} \delta_{(0,1)} + \frac{9}{10} \delta_{(0,0)})
= \frac{91}{100} \delta_{(0,0)} + \frac8{100} \delta_{(0,1)} + \frac1{100} \delta_{(1,0)}.
\end{align*}
Moreover, the law of $\theta$ is given by $\frac{19}{20} \delta_{(0,0)}+\frac1{20} \delta_{(1,0)}$.
The variables $J_k$ are irrelevant in the given situation as step sizes here are at most one,
hence there will never be the situation where a visitor attempts to buy more of the popular good
than what is left.
We conclude that the theoretical conversion rates are given by
\begin{equation*}	\textstyle
p_0 = \mu_0(\{(0,0)\}^\comp) = \frac{52}{1\,000},
\quad
p_1 = \mu_1(\{(0,0)\}^\comp) = \frac{9}{100}
\quad	\text{and}	\quad
p_\theta = \frac{5}{100}.
\end{equation*}
We shall see that if $c_n = c \sqrt{n}$, then the chi-squared test will reject the hypothesis with probability tending to $1$
as $c$ becomes large.
On the other hand, we shall demonstrate that Algorithm $2$ does not perform better
given the model assumptions (A1) through (A5).
\end{example}

\section{Testing for the higher conversion rate}

We address the question which algorithm, when used alone, leads to the higher conversion rate,
where the conversion rate is the total number of sales divided by the total number of visitors.
More formally, for $i=0,1$, we define
\begin{align*}
N_n^{(i)}	\defeq \sum_{k=1}^n \1_{\{I_k=i\}}
\quad	\text{and}	\quad
L_n^{(i)}	\defeq \sum_{k=1}^n \1_{\{Z_k > 0,\, I_k=i\}},
\end{align*}
which model the number of visitors of website version $i$
and the number of those visitors who make a purchase.
We set $L_n \defeq L_n^{(0)}+L_n^{(1)}$, which is the total number of purchases,
and notice that $N_n^{(0)} + N_n^{(1)} = n$.
Then
\begin{equation*}
C_n^{(i)} \defeq \tfrac{L_n^{(i)}}{N_n^{(i)}}
\end{equation*}
is the empirical conversion rate in group $i$.
We stipulate that $C_n^{(i)} \defeq 0$ on $\{N_n^{(i)}=0\}$.

If one chooses $p \in \{0,1\}$, then $C_n^{(p)}$ under $\Prob_p$ is the empirical conversion rate
when only website $i=p$ is used.
In view of this, version $1$ of the website is better than version $0$
if $C_n^{(1)}$ under $\Prob_1$ is `larger' than $C_n^{(0)}$ under $\Prob_0$.
Here, the term `larger' is not specified a priori, so we need to clarify what we mean by this.

\subsection{The chi-squared test statistics in the classical model}	\label{subsec:limiting law classical}

In the classical model, i.e., if assumptions ($\chi1$), ($\chi2$) and ($\chi3$) are in force,
$Z_1, Z_2, \ldots$ are i.i.d.\ 
with expectation $\E[Z_1] = \Prob(Z_1 = 1) = pp_1+(1-p)p_0$.
Hence, by the strong law of large numbers, for $p \in \{0,1\}$,
\begin{equation*}
c^{(p)} \defeq \lim_{n \to \infty} C_n^{(p)} = \lim_{n \to \infty} \tfrac{1}{n}L_n^{(p)}
= \Prob(\xi+\eta > 0) = p p_1 + (1-p) p_0	\quad	\text{a.\,s.}
\end{equation*}
Consequently, testing whether $C_n^{(1)}$ under $\Prob_1$ is `different' from $C_n^{(0)}$ under $\Prob_0$
can be formulated as follows:
\begin{equation*}
H_0: p_1 = p_0	\qquad	\text{vs.}	 \qquad	H_1: p_1 \not= p_0.
\end{equation*}
More interest, in fact, would be in the corresponding one-sided test problem.
Hence, in this case, it is a classical test problem and widely used tests for this problem are
the chi-squared test, the $G$-test, and Fisher's exact test.
To keep the paper short, we shall always restrict attention to the chi-squared test.
For the reader's convenience, we recall some facts about this test.

\begin{table}[htp]
\caption{Contingency table}
\begin{center}
\begin{tabular}{|c|c|c|c|}
\hline
		&	purchase				&	no purchase			&	$\sum$		\\	\hline
group 0	&	$L_n^{(0)}	$			&	$N_n^{(0)}-L_n^{(0)}$	&	$N_n^{(0)}$	\\	\hline
group 1	&	$L_n^{(1)}	$			&	$N_n^{(1)}-L_n^{(1)}$	&	$N_n^{(1)}$	\\	\hline
$\sum$	&	$L_n$				&	$n-L_n$				&	$n$			\\	\hline
\end{tabular}
\end{center}
\label{default}
\end{table}%

\noindent
The test statistics for the chi-squared test is
\begin{equation}\label{eq:chi_square}
\chi^2	\defeq	 \sum_{i=0,1} \frac{\big(L_n^{(i)} - L_n \frac{N_n^{(i)}}{n}\big)^2}{L_n \frac{N_n^{(i)}}{n}}
+ \sum_{i=0,1} \frac{\big(N_n^{(i)}\!-\!L_n^{(i)}-(n\!-\!L_n)\frac{N_n^{(i)}}{n}\big)^2}{(n\!-\!L_n) \frac{N_n^{(i)}}{n}}.
\end{equation}
If ($\chi$1) through ($\chi3$) are in force, as $n \to \infty$, the distribution of $\chi^2$
approaches a chi-squared distribution with $1$ degree of freedom \cite[Chapter 17]{vdVaart:1998}.
Write $q_{1-\alpha}$ for the $1-\alpha$ quantile of the chi-squared distribution with $1$ degree of freedom.
Then, with significance level of $\alpha$, the hypothesis is rejected if $\chi^2 > q_{1-\alpha}$.

\subsection{The limiting law of the chi-squared test statistics in the model incorporating low inventory}	\label{subsec:limiting law}

Now suppose that there is a rare but popular good, i.e., suppose that
the model assumptions (A1) through (A5) hold.
One goal of this paper is to point out in a quantitative way that when (A1) through (A5) instead of ($\chi1$) through ($\chi3$) are in force,
then the chi-squared test may produce too many false positives.
In other words, it may fail to hold the specified significance level.
This is because the distribution of $\chi^2$ under the null hypothesis
is different when (A1) through (A5) rather than ($\chi1$) through ($\chi3$) are in force.
The detailed statement is given in the following theorem.

\begin{theorem}	\label{Thm:cvgce chi^2 statistics}
Suppose that (A1) through (A5) and \eqref{eq:c_n reg var} are in force.
Assume additionally that
\begin{equation*}
d_\infty \defeq \lim_{n\to\infty}\frac{c_n}{\sqrt{n}}\in [0,\infty).
\end{equation*}
Then the chi-squared statistics defined by \eqref{eq:chi_square} satisfies
\begin{align*}	\textstyle
\chi^2
&\distto
\left(\Normal-\frac{d_\infty(p_0-p_1)\sqrt{p(1-p)}}{m^{\eta}\sqrt{p_\theta(1-p_\theta)}}\right)^{\!\! 2}
\qquad	\text{as }	n \to \infty
\end{align*}
where $\Normal$ is a standard normal random variable.
\end{theorem}

Hence, if one applies the chi-squared test with significance level $\alpha \in (0,1)$ in the given situation,
the test rejects the hypothesis with (asymptotic) probability
\begin{align*}	\textstyle
\Prob\left(\left(\Normal-\frac{d_\infty(p_0-p_1)\sqrt{p(1-p)}}{m^{\eta}\sqrt{p_\theta(1-p_\theta)}}\right)^{\!\! 2} > q_{1-\alpha}\right)
> \Prob\left(\Normal^{2} > q_{1-\alpha}\right) = \alpha.
\end{align*}
In fact, the probability on the left-hand side tends to $1$ as $d_\infty \to \infty$.
We specialize to the situation of Example \ref{Exa:ranking experiment}.

\begin{example}[Ranking experiment, take 3]	\label{Exa:main theorem in ranking experiment}
Recall the situation of Example \ref{Exa:ranking experiment}.
Then, see also Example \ref{Exa:ranking experiment as special case of our model},
we have
\begin{equation*}	\textstyle
p = \frac12,
\qquad	p_\theta = \frac1{20},
\qquad	p_0 = \frac{52}{1\,000},	\qquad p_1 = \frac{9}{100},
\qquad	m^\eta=\frac{42}{1\,000}
\qquad	\text{and}	\qquad
d_\infty = \frac12.
\end{equation*}
Consequently,
\begin{equation}	\label{eq:parameters in ranking experiment}
\frac{d_\infty(p_0-p_1)\sqrt{p(1-p)}}{m^{\eta}\sqrt{p_\theta(1-p_\theta)}}
= \frac{\frac12 \cdot \frac{-38}{1\,000} \cdot \frac12}{\frac{42}{1\,000} \cdot \sqrt{\frac{19}{400}}}
= - \frac{ \sqrt{19} \cdot 5}{21}
= -1.037833\ldots
\end{equation}
Hence, in the given situation, the chi-squared test rejects the hypothesis with (asymptotic) probability
\begin{align*}	\textstyle
\Prob\left((\Normal+1.037833)^{2} > q_{95\%}\right)
= 0.1795898\ldots > 0.05.
\end{align*}
This becomes worse as $d_\infty$ becomes larger, see Figure \ref{fig:false_positives} below.
\begin{figure}[h]
\begin{center}
\includegraphics[width=0.3\textwidth]{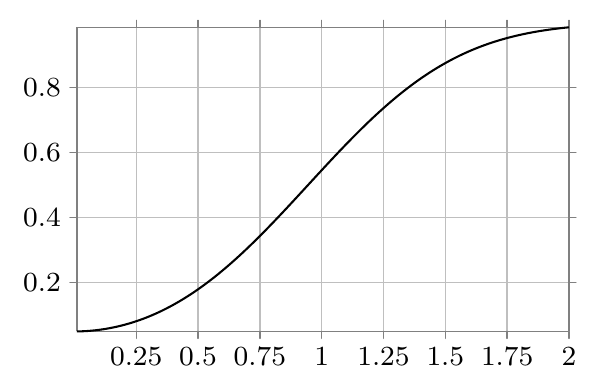}
\caption{False-positive probability as a function of $d_\infty$ with all other parameters fixed as in \eqref{eq:parameters in ranking experiment}.}
\label{fig:false_positives}
\end{center}
\end{figure}

At first glance, this may occur to be no problem as $p_1 > p_0$,
so one is tempted to guess that algorithm $1$ performs better than algorithm $0$
and what we see above is just the power of the test, which becomes better as $d_\infty$ becomes large.
However, we shall argue in Example \ref{Exa:same performance in ranking experiment} below
that the two algorithms perform equally well when used separately.
\end{example}

Next, we show that in the general situation,
assuming that (A1) through (A5) hold and that $\frac{c_n}{n} \to 0$,
we show that on the linear scale, the asymptotic empirical conversion rates of the two versions of the website,
when used separately, are identical.

\begin{theorem}	\label{Thm:SLLN C_n^(p)}
Suppose that (A1) through (A5) hold and that $c_\infty \defeq \lim_{n \to \infty} \frac{c_n}{n} = 0$.
Then, for $p \in \{0,1\}$,
\begin{equation*}
\lim_{n \to \infty} C_n^{(p)} = p_\theta \quad	\Prob_p\text{-a.\,s.}
\end{equation*}
\end{theorem}
\begin{proof}
The result is a consequence of Theorem \ref{Thm:SLLN L_n's and N_n's} below.
\end{proof}

Hence, in the relevant regime ($\frac{c_n}{n} \to 0$) the first order of growth
of $C_n^{(p)}$ depends only on what happens after the popular good is sold out.
According to our model assumptions, the two versions of the website have identical
performance once the popular good is sold out.
This implies that on the linear scale, there is no difference between the two versions of the website.
Hence, we need to make a comparison on a finer scale.

\subsection{A joint limit theorem for the group sizes and numbers of purchases in each group}

As $\chi^2$ is a function of $(N_n^{(0)},N_n^{(1)},L_n^{(0)},L_n^{(1)})$,
a limit theorem for $\chi^2$ follows from one for the above vector via the continuous mapping theorem \cite[Theorem 2.7]{Billingsley:1999}.
We begin with a strong law of large numbers for the variables $L_n^{(0)}$ and $L_n^{(1)}$
(the corresponding result for $N_n^{(0)}$ and $N_n^{(1)}$ is classical).

\begin{theorem}	\label{Thm:SLLN L_n's and N_n's}
Suppose that (A1) through (A5) are in force and that the limit
\begin{equation*}
c_{\infty}\defeq\lim_{n\to\infty}\frac{c_n}{n}\in [0,\infty)
\end{equation*}
exists\footnote{In the applications we have in mind, $c_{\infty}=0$
because the quantity of good 2 should be much smaller than the total number of observations.
But from a theoretical perspective positive values of $c_{\infty}$ are also interesting
because of the occurrence of different asymptotic regimes.
Let us also stress that $c_{\infty}>0$ necessitates $\rho=1$,
where the definition of $\rho$ may be recalled from \eqref{eq:c_n reg var}.}. If $c_{\infty}\in [0,\frac{1}{m^{\eta}})$, then
\begin{align*}
\tfrac1n L_n^{(0)} \to (1-p)p_\theta+c_{\infty}(1-p)\tfrac{1}{m^{\eta}}(p_0-p_\theta)	
\quad	\text{and}	\quad 
\tfrac1n L_n^{(1)} \to pp_\theta+c_{\infty}p\tfrac{1}{m^{\eta}}(p_1-p_\theta)	\quad	\text{a.\,s.}
\end{align*}
In particular, in the most relevant case $c_\infty=0$,
\begin{align*}
\tfrac1n L_n^{(0)} \to (1-p)p_\theta	
\quad	\text{and}	\quad 
\tfrac1n L_n^{(1)} \to pp_\theta	\quad	\text{a.\,s.}
\end{align*}
If $c_{\infty}>\frac{1}{m^{\eta}}$, then
\begin{align*}
\tfrac1n L_n^{(0)} \to (1-p)p_0	
\quad\text{and}\quad 
\tfrac1n L_n^{(1)} \to pp_1	\quad	\text{a.\,s.}
\end{align*}
\end{theorem}

We continue with the asymptotic law of the vector $(N_n^{(0)},N_n^{(1)},L_n^{(0)},L_n^{(1)})$,
suitably shifted and scaled in the most relevant scenario where $c_n$ is of the order $\sqrt{n}$.

\begin{theorem}	\label{Thm:joint cvgce var in chi^2 statistics}
Suppose that (A1) through (A5) are in force and suppose in addition to \eqref{eq:c_n reg var} that the limit
\begin{equation*}
d_\infty = \lim_{n\to\infty}\frac{c_n}{\sqrt{n}}\in [0,\infty)
\end{equation*}
exists, we have, as $n\to\infty$,
\begin{multline}	\label{eq:cor weak convergence2}
\bigg(\frac{N^{(0)}_{n}-(1-p)n}{\sqrt{n}},\frac{N^{(1)}_{n}-pn}{\sqrt{n}},\frac{L_{n}^{(0)}-n(1-p)p_\theta}{\sqrt{n}},\frac{L_{n}^{(1)}-npp_\theta}{\sqrt{n}}\bigg)\\
\distto d_\infty \big(0,0,(1\!-\!p)\tfrac{1}{m^\eta}(p_0\!-\!p_\theta),p\tfrac{1}{m^\eta}(p_1\!-\!p_\theta)\big)+ (G_1,-G_1,G_2,G_3)
\end{multline}
where $\left(G_1,G_2,G_3\right)$ is a centered Gaussian vector with covariance matrix 
\begin{equation}	\label{eq:covariance_matrix_v1}
\mathbf{V_1}=\left(
\begin{matrix}
p(1\!-\!p)			&  p(1\!-\!p)p_\theta					& -p(1\!-\!p)p_\theta		\\
p(1\!-\!p)p_\theta	& p_\theta(1\!-\!p)(1\!-\!p_\theta(1\!-\!p))	& -p(1\!-\!p)p_\theta^2	\\
-p(1\!-\!p)p_\theta	& -p(1\!-\!p)p_\theta^2				& pp_\theta(1\!-\!pp_\theta)\\
\end{matrix}\right).
\end{equation}
\end{theorem}

Notice that the theorem contains a limit theorem for the pure conversion rates $C_n^{(0)}$ and $C_n^{(1)}$
by choosing $p=0$ and projecting on the third coordinate or by choosing $p=1$ and projecting on the fourth component, respectively.
This gives, with $\Normal(\mu,\sigma^2)$ denoting a normal random variable with mean $\mu$ and variance $\sigma^2$,
\begin{align}	
\frac{L_{n}^{(0)}-np_\theta}{\sqrt{n}} &\distto \Normal(d_\infty \tfrac{p_0\!-\!p_\theta}{m_0^\eta},p_\theta(1\!-\!p_\theta))
\quad	\text{under }	\Prob_0	\label{eq:limit theorem for C_n^0}	\\
\text{and}\qquad
\frac{L_{n}^{(1)}-np_\theta}{\sqrt{n}} &\distto \Normal(d_\infty \tfrac{p_1\!-\!p_\theta}{m_1^\eta},p_\theta(1\!-\!p_\theta))
\quad	\text{under }	\Prob_1	\label{eq:limit theorem for C_n^1}
\end{align}
where $m^\eta = pm^\eta_1+(1-p)m^\eta_0$ with $p=0$ and $p=1$, respectively, has been used.
Hence, if the two expectations in \eqref{eq:limit theorem for C_n^0} and \eqref{eq:limit theorem for C_n^1}
coincide, then the performances of the two websites coincide asymptotically both on the linear scale as well as on the level of fluctuations.
The subsequent example demonstrates that this can be the case even if $p_0 \not = p_1$.

\begin{example}[Ranking experiment, take 4]	\label{Exa:same performance in ranking experiment}
Recall the situation of Example \ref{Exa:ranking experiment} and the calculations of Example \ref{Exa:main theorem in ranking experiment}:
\begin{equation*}	\textstyle
p_\theta = \frac1{20},
\qquad	p_0 = \frac{52}{1\,000},	\qquad p_1 = \frac{9}{100},
\qquad	m_0^\eta=\frac{4}{1\,000},
\qquad	m_1^\eta=\frac{8}{100}
\qquad	\text{and}	\qquad
d_\infty = \frac12.
\end{equation*}
Consequently,
\begin{equation*}
\tfrac{p_0\!-\!p_\theta}{m_0^\eta} = \frac{\frac2{1\,000}}{\frac{4}{1\,000}} = \frac12
=\frac{\frac4{100}}{\frac{8}{100}} = \tfrac{p_1\!-\!p_\theta}{m_1^\eta}.
\end{equation*}
This means that the two limits in \eqref{eq:limit theorem for C_n^0} and \eqref{eq:limit theorem for C_n^1}
coincide in the given example even though $p_1 > p_0$.
\end{example}

From Theorem \ref{Thm:joint cvgce var in chi^2 statistics}, we may immediately deduce a limit theorem for $\chi^2$,
which is a preliminary version of Theorem \ref{Thm:cvgce chi^2 statistics}.

\begin{corollary}	\label{Cor:cvgce chi^2 statistics}
Suppose that (A1) through (A5) and \eqref{eq:c_n reg var} are in force and that the limit
\begin{equation*}
d_\infty = \lim_{n\to\infty}\frac{c_n}{\sqrt{n}}\in [0,\infty)
\end{equation*}
exists, we have the following limit theorem for the chi-squared statistics defined by \eqref{eq:chi_square}
\begin{align*}
\chi^2 &\distto
\frac{(d_2+G_2-p_\theta G_1-(1-p)(d_2+d_3+G_2+G_3))^2}{(1-p_\theta)p_\theta(1-p)}	\\
&\hphantom{\distto\ }+\frac{(d_3+G_3+p_\theta G_1-p(d_2+d_3+G_2+G_3))^2}{(1-p_\theta)p_\theta p}
\qquad	\text{as }	n \to \infty
\end{align*}
where $(0,0,d_2,d_3)$ denotes the expectation of the right-hand side in \eqref{eq:cor weak convergence2}
and $(G_1,G_2,G_3)$ is the Gaussian vector from Theorem \ref{Thm:joint cvgce var in chi^2 statistics}.
\end{corollary}

We close this section with another example showing that ignoring the dependencies might also lead
to a high false negative rate.

\begin{example}[Ranking experiment with picky customers]	\label{Exa:picky customers}
We consider a variant of Example \ref{Exa:ranking experiment}
in which there are picky customers that will only buy the rare good.
This time, we use the former Algorithm $1$ from above as the default algorithm displaying the rare goods first.
The former Algorithm $0$ strategically keeps the rare inventory for later arrival of picky customers
by ranking the rare good low.
In an experiment with shared inventory, Algorithm $1$ sells off the rare good greedily,
and the value of Algorithm $0$'s strategy will not be properly assessed in the model ignoring dependencies.
We make this precise in the following.

Again suppose that during a test phase, $n=4\,000\,000$ customers visit the website.
Again, there are $c_n = 1\,000$ rare goods while good $1$ is available at sufficient quantities.

The algorithms work as in Example \ref{Exa:ranking experiment}, but there is a difference
in the behavior of the customers.
We assume that, independent of all other customers,
each customer has a $1\%$ chance of being picky.
If not picky, the customer behaves like the customers of Example \ref{Exa:ranking experiment}.
A picky customer, however, will search as long as required to check whether
there is something of the rare good left. If the rare good is still available,
the picky customer buys one unit with $50\%$ chance.
Otherwise, the customer leaves the website.
The overviews given in Figure \ref{fig:Algo 1} still apply to regular customers,
for picky customers and when good 2 is still available, there is a simplified decision tree:
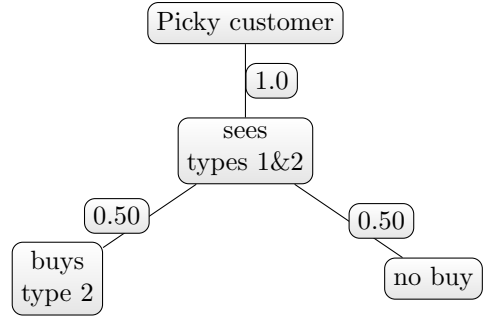
\begin{figure}[h]
\begin{tikzpicture}[level distance=1.7cm,level 1/.style={sibling distance=7cm},
level 2/.style={sibling distance=5cm},
level 3/.style={sibling distance=3.2cm},
  every node/.style = {shape=rectangle, rounded corners,
    draw, align=center,
    top color=white, bottom color=black!10}]
  \node {Picky customer}
        child {node {sees\\ types $1\&2$}
        	child{node {buys\\ type $2$}
     	edge from parent node[left] {$0.50$}}
        	child{node {no buy}
        edge from parent node[right] {$0.50$}}
edge from parent node[right] {$1.0$}};
\end{tikzpicture}
	\caption{Ranking experiment: Picky customer behavior when good $2$ is still available}
	\label{fig:Picky customer}
\end{figure}

\noindent
A calculation of the relevant model parameters in this example
based on the corresponding parameter values from Example \ref{Exa:main theorem in ranking experiment}
gives $p=\frac12$ and $d_\infty = \frac12$ as before and
\begin{align*}	\textstyle
p_\theta = \frac{99}{100} \cdot \frac1{20},
\quad 	p_0 = \frac{99}{100} \cdot\frac{52}{1\,000} + \frac1{100} \cdot \frac12,
\quad  	p_1 = \frac{99}{100} \cdot \frac{9}{100} + \frac1{100} \cdot \frac12,
\quad 	m^\eta = \frac{99}{100} \cdot \frac{42}{1\,000} + \frac1{100} \cdot \frac12.
\end{align*}
We shall now show that Algorithm $0$ performs actually better than Algorithm $1$.
We have
\begin{align*}	\textstyle
m_0^\eta=\frac{99}{100} \cdot\frac{4}{1\,000} + \frac1{100} \cdot \frac12
\quad	\text{and}	\quad
m_1^\eta=\frac{99}{100} \cdot\frac{8}{100}  + \frac1{100} \cdot \frac12.
\end{align*}
Consequently,
\begin{equation*}
d_\infty \cdot \tfrac{p_0\!-\!p_\theta}{m_0^\eta} = \frac{349}{896} > 
\frac{223}{842} = d_\infty \cdot \tfrac{p_1\!-\!p_\theta}{m_1^\eta}.
\end{equation*}
In view of \eqref{eq:limit theorem for C_n^0} and \eqref{eq:limit theorem for C_n^1},
Algorithm $0$ does perform better than Algorithm $1$.
Now let us calculate the probability that the chi-squared test
rejects the hypothesis that Algorithm $0$ and Algorithm $1$ perform equally well.
To this end, we first calculate
\begin{align*}
\frac{d_\infty (p_1-p_0) \sqrt{p(1-p)}}{m^\eta \sqrt{p_\theta(1-p_\theta)}}
= 0.930852\ldots.
\end{align*}
According to Theorem \ref{Thm:cvgce chi^2 statistics},
the chi-squared test (with significance level $5\%$) rejects the hypothesis with probability $0.1536348\dotsc$.
However, it is a standard practice to say that Algorithm $0$ is significantly better than Algorithm $1$
only when $\chi^2 > q_{1-\alpha}$ and $C_n^{(0)} > C_n^{(1)}$ (with $\alpha \in (0,1)$ being the significance level).
So, asymptotically, the power of the test is $\lim_{n \to \infty} \Prob(C_n^{(0)} > C_n^{(1)}, \chi^2 > q_{1-\alpha})$.
We shall now calculate this probability in the given situation with $d_\infty = 1/2$ but also as a function
of $d_\infty$ to point out that the probability becomes arbitrarily small as $d_\infty$ becomes large.
We begin by reformulating the condition $C_n^{(0)} > C_n^{(1)}$.
Recall that $C_n^{(i)} = L_n^{(i)}/N_n^{(i)}$ for $i=0,1$. Hence,
\begin{align*}
&C_n^{(0)} > C_n^{(1)}
\quad	\text{iff}	\quad
L_n^{(0)} N_n^{(1)} - L_n^{(1)} N_n^{(0)} > 0	\\
&\text{iff}	\quad
(L_n^{(0)}-\tfrac{n p_\theta}2) (N_n^{(1)}-\tfrac{n}2) - (L_n^{(1)}-\tfrac{n p_\theta}2) (N_n^{(0)}-\tfrac{n}2)
+ \tfrac{np_\theta}{2} (N_n^{(1)}-N_n^{(0)})
+ \tfrac{n}2 (L_n^{(0)}-L_n^{(1)}) > 0 	\\
&\text{iff}	\quad
\frac2{\sqrt{n}} \frac{L_n^{(0)}-n p_\theta/2}{\sqrt{n}} \cdot \frac{N_n^{(1)}-n/2}{\sqrt{n}}
- \frac2{\sqrt{n}} \frac{L_n^{(1)}- n p_\theta/2}{\sqrt{n}} \cdot \frac{N_n^{(0)}-n/2}{\sqrt{n}}	\\
&\hphantom{\text{iff}}	\quad
- p_\theta \frac{N_n^{(0)} - n/2}{\sqrt{n}} + p_\theta \frac{N_n^{(1)} - n/2}{\sqrt{n}}
+ \frac{L_n^{(0)}-np_\theta/2}{\sqrt{n}} - \frac{L_n^{(1)}-n p_\theta/2}{\sqrt{n}}  > 0.
\end{align*}
By Theorem \ref{Thm:joint cvgce var in chi^2 statistics} and Slutsky's theorem,
the two terms in the penultimate line tend to $0$ in probability as $n \to \infty$.
By Theorem \ref{Thm:joint cvgce var in chi^2 statistics}, the other summands converge in distribution
so that in the limit, the above inequality becomes
\begin{align*}
-2 p_\theta G_1 + d_\infty \tfrac{p_0\!-\!p_\theta}{2 m^\eta} + G_2 - d_\infty \tfrac{p_1\!-\!p_\theta}{2 m^\eta} - G_3 > 0,
\end{align*}
which can be simplified to
\begin{align*}
-2 p_\theta G_1 + G_2 - G_3 > d_\infty \tfrac{p_1\!-\!p_0}{2 m^\eta}.
\end{align*}
By Theorem \ref{Thm:cvgce chi^2 statistics}, $\chi^2$ converges also in distribution.
According to Corollary \ref{Cor:cvgce chi^2 statistics}, we can express the condition
$\chi^2>q_{1-\alpha}$ in the limit as $n \to \infty$ in the form
\begin{align}
(d_2+G_2-p_\theta G_1-\tfrac12(d_2+d_3+G_2+G_3))^2
+ (d_3+G_3+p_\theta G_1-\tfrac12(d_2+d_3+G_2+G_3))^2	\notag	\\
> \tfrac12 (1-p_\theta)p_\theta q_{1-\alpha}		\label{eq:chi^2 > q_1-alpha in terms of Gs}
\end{align}
where $(0,0,d_2,d_3)$ is the expectation vector in \eqref{eq:cor weak convergence2},
i.e.,
$(0,0,d_2,d_3)
= d_\infty \big(0,0,\tfrac{p_0\!-\!p_\theta}{2 m^\eta},\tfrac{p_1\!-\!p_\theta}{2m^\eta}\big)$.
Since the convergence in Theorem \ref{Thm:joint cvgce var in chi^2 statistics} is jointly
and since the law of $(G_1,G_2,G_3)$ on $\R^3$
is absolutely continuous with respect to Lebesgue measure, the Portmanteau theorem
implies that
\begin{align*}
\lim_{n \to \infty} \Prob(\chi^2 > q_{1-\alpha}, C_n^{(0)} > C_n^{(1)})
= \Prob\big(\text{\eqref{eq:chi^2 > q_1-alpha in terms of Gs} holds and }-2 p_\theta G_1 + G_2 - G_3 > d_\infty \tfrac{p_1\!-\!p_0}{2 m^\eta}\big).
\end{align*}
In the given situation (with $d_\infty = \frac12$), we have used a Monte Carlo simulation (with $2 \cdot 10^6$ iterations)
to estimate the power of the test resulting in an estimate of $0.001933\ldots$
Further, it is already immediate that the power tends to $0$ as $d_\infty$ tends to $\infty$
since $p_1 > p_0$.
To get a better quantitative picture, we have performed Monte Carlo simulations
for 200 equidistant values of $d_\infty$ between $0$ and $1$, each with $2 \cdot 10^6$ iterations.
The results of the Monte Carlo simulations are displayed in Figure \ref{fig:false_negatives}.
\begin{figure}[h]
\begin{center}
\includegraphics[width=0.5\textwidth]{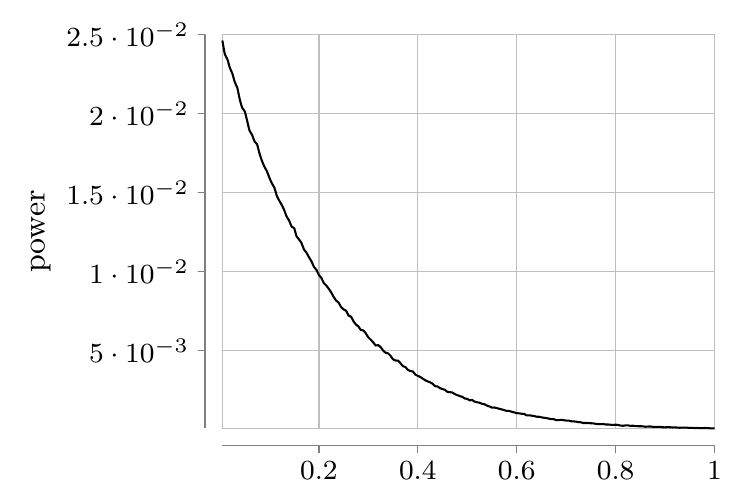}
\caption{Monte Carlo simulation of the power of the test as a function of $d_\infty$ with all other parameters fixed as above.}
\label{fig:false_negatives}
\end{center}
\end{figure}

\noindent
In particular, for every value $d_\infty > 0$, the estimate of the power of the test is strictly smaller than $0.025 = \frac\alpha2$.
\end{example}

\subsection{Functional limit theorem for the model incorporating low inventory}	\label{subsec:FCLTs}

We deduce Theorem \ref{Thm:joint cvgce var in chi^2 statistics} from a more general result,
namely, a joint functional central limit theorem.
To formulate it, we need additional notation.

Henceforth, convergence in distribution of random elements in the Skorohod spaces
$D([0,\infty),\R^d)$ and $D((0,\infty),\R^d)$ of $\R^d$-valued, right-continuous functions with existing left limits
is with respect to the standard $J_1$-topology
and will be denoted by $\Longrightarrow$.
To distinguish between convergence in the above two spaces we adopt the following convention.
The convergence is in $D([0,\infty),\R^d)$ if the processes are written with subscript $(\cdot)_{t \geq 0}$,
whilst if the subscript is $(\cdot)_{t > 0}$, the convergence is in $D((0,\infty),\R^d)$.

Let $({\bf B}(t))_{t\geq 0}=\left((B_1(t),\ldots,B_7(t))\right)_{t\geq 0}$ be a centered $7$-dimensional Brownian motion with covariance matrix
\begin{multline}	\label{eq:covariance matrix V}
\mathbf{V}=\left(
\begin{matrix}
p_0(1\!-\!p)(1\!-\!p_0(1\!-\!p))		& -p_0p_1p(1\!-\!p)		& -p_0(1\!-\!p)p			& (1\!-\!p)(m_0^{\xi}-m^{\xi}p_0)\\
-p_0p_1p(1\!-\!p)				& pp_1(1\!-\!pp_1)		& pp_1(1\!-\!p)			& p(m_1^{\xi}-p_1m^{\xi}) \\
-p_0(1\!-\!p)p					& pp_1(1\!-\!p)			& p(1\!-\!p)			& p(m_1^{\xi}-m^{\xi}) \\
(1\!-\!p)(m_0^{\xi}-m^{\xi}p_0)		& p(m_1^{\xi}-p_1m^{\xi})	& p(m_1^{\xi}-m^{\xi})	& \sigma^2_{\xi}	\\
(1\!-\!p)(m_0^{\eta}-m^{\eta}p_0)	& p(m_1^{\eta}-p_1m^{\eta})& p(m_1^{\eta}-m^{\eta})& \rho_{\xi\eta} \\
p_0(1\!-\!p)pp_\theta				& -p_1p(1\!-\!p)p_\theta	& -p(1\!-\!p)p_\theta		& (1\!-\!p)p_\theta(m_0^{\xi}-m^{\xi}) \\ 
-p_0p(1\!-\!p)p_\theta			& pp_1(1\!-\!p)p_\theta	& p(1\!-\!p)p_\theta		& pp_\theta(m_1^{\xi}-m^{\xi}) \\
\end{matrix}\right.\\
\left.\begin{matrix}
(1\!-\!p)(m_0^{\eta}-m^{\eta}p_0)	& p_0(1\!-\!p)pp_\theta				& -p_0p(1\!-\!p)p_\theta\\
p(m_1^{\eta}-p_1m^{\eta})			& -p_1p(1\!-\!p)p_\theta				& pp_1(1\!-\!p)p_\theta\\
p(m_1^{\eta}-m^{\eta})			& -p(1\!-\!p)p_\theta					& p(1\!-\!p)p_\theta\\
\rho_{\xi\eta}					& (1\!-\!p)p_\theta(m_0^{\xi}-m^{\xi})		& pp_\theta(m_1^{\xi}-m^{\xi})\\
\sigma_{\eta}^2					& (1\!-\!p)p_\theta(m_0^{\eta}-m^{\eta})	& pp_\theta(m_1^{\eta}-m^{\eta})\\
(1\!-\!p)p_\theta(m_0^{\eta}-m^{\eta})& p_\theta(1\!-\!p)(1\!-\!p_\theta(1\!-\!p))	& -p(1\!-\!p)p_\theta^2\\
pp_\theta(m_1^{\eta}-m^{\eta})		& -p(1\!-\!p) p_\theta^2				& pp_\theta(1\!-\!pp_\theta)\\
\end{matrix}\right),
\end{multline}
that is, ${\bf B}(t)={\bf V}^{1/2}{\bf B}^{\prime}(t)$, $t\geq 0$ where ${\bf B}^{\prime}(t)$ is a $7$-dimensional Brownian motion
with independent components each being a one-dimensional standard Brownian motion,
and ${\bf V}^{1/2}$ is the square root of the positive semi-definite matrix ${\bf V}$.

\begin{theorem}	\label{Thm:joint CLT}
Suppose that (A1) through (A5) and \eqref{eq:c_n reg var} are in force and that the limit
\begin{equation*}
c_\infty = \lim_{n\to\infty} \frac{c_n}{n} \in [0,\infty)
\end{equation*}
exists.
If $c_\infty \in [0,\frac{1}{m^{\eta}})$, then, as $n \to \infty$,
\begin{align}
&\Bigg(\frac{N^{(0)}_{\lfloor nt\rfloor}-(1-p)nt}{\sqrt{n}},\frac{N^{(1)}_{\lfloor nt\rfloor}-pnt}{\sqrt{n}},\frac{L_{\lfloor nt\rfloor}^{(0)}-(1-p)\tfrac{1}{m^\eta}(p_0 -p_\theta) c_{\lfloor nt\rfloor}-nt(1-p)p_\theta}{\sqrt{n}},	\notag	\\
&\hphantom{\Bigg(\frac{N^{(0)}_{\lfloor nt\rfloor}-(1-p)nt}{\sqrt{n}},}\frac{L_{\lfloor nt\rfloor}^{(1)} - p\tfrac{1}{m^\eta}(p_1-p_\theta)c_{\lfloor nt\rfloor}-nt pp_\theta}{\sqrt{n}}\Bigg)_{t> 0}\notag\\
&\Longrightarrow \Big(-B_3(t),B_3(t),c_\infty^{\frac12}\big(B_1\big(\tfrac{t}{m^{\eta}}\big)+(1\!-\!p)(p_\theta-p_0)(m^{\eta})^{\frac32}B_5(t)-B_6\big(\tfrac{t}{m^{\eta}}\big)\big)+B_6(t),\notag\\ 
&\hphantom{\Longrightarrow \big(-B_3(t),B_3(t),}c_\infty^{\frac12}\big(B_2\big(\tfrac{t}{m^{\eta}}\big)+p(p_\theta-p_1)(m^{\eta})^{\frac32}B_5(t)-B_7\big(\tfrac{t}{m^{\eta}}\big)\big)+B_7(t)\Big)_{\! t > 0}\!\!.	\label{eq:joint FCLT small c_n}
\end{align}
On the other hand, if $c_{\infty}>\frac1{m^{\eta}}$, then, as $n \to \infty$,
\begin{align}
\Bigg(\frac{N^{(0)}_{\lfloor nt\rfloor}-(1-p)nt}{\sqrt{n}},\frac{N^{(1)}_{\lfloor nt\rfloor}-pnt}{\sqrt{n}},\frac{L_{\lfloor nt\rfloor}^{(0)}-(1-p)p_0 nt}{\sqrt{n}},\frac{L_{\lfloor nt\rfloor}^{(1)}	- pp_1 nt}{\sqrt{n}}\Bigg)_{t> 0}	\notag	\\
\Longrightarrow \left(-B_3(t),B_3(t),B_1(t),B_2(t)\right)_{t > 0}.	\label{eq:joint FCLT large c_n}
\end{align}
\end{theorem}


\section{Proofs}

In the model, there are two natural breaks, namely,
first the process evolves in the positive quadrant like
an unrestricted two-dimensional random walk until the second coordinate of the walk
for the first time attempts to step to or beyond the level $c_n$.
This time we call $\tau_1(n)$.
Then there is a number of attempts to reach that border
until this is eventually achieved at a time we call $\tau_2(n)$.
From that time on, the walk keeps the second coordinate fixed at $c_n$
and evolves horizontally as a one-dimensional walk.
Crucial both for the proof of the strong law of large numbers, Theorem \ref{Thm:SLLN L_n's and N_n's},
and the joint functional limit theorem, Theorem \ref{Thm:joint CLT},
is a sufficient understanding of these times,
$\tau_1(n)$ and $\tau_2(n)$.

\subsection{Stopping time analysis and the strong law of large numbers}	\label{subsec:stopping time analysis}

Formally, the two natural stopping times associated with the stochastic process $(\bR_k)_{k \in \N_0}$
are defined as follows:
\begin{itemize}
	\item $\tau_1(n)\defeq\inf\{k \in \N: T_{k-1}+\eta_k\geq c_n\}$,
		the first attempt to reach the half-plane $\N_0 \times [c_n,\infty)$, and
	\item $\tau_2(n)\defeq\inf\{k \in \N: T_k=c_n\}$,
		the first entrance to the horizontal line $\N_0 \times \{c_n\}$.
\end{itemize}
We stipulate that $\tau_1(0)=\tau_2(0)=0$.
To simplify notation later on, we set $\tau_j(t) \defeq \tau_j(\lfloor t \rfloor)$
for $t \geq 0$ and $j=1,2$.
The above quantities are stopping times with respect to the natural filtration of $(\bR_k)_{k \in \N_0}$.
By the strong law of large numbers, we have
\begin{equation}	\label{eq:SLLN tau_1(n)}
\frac{\tau_1(n)}{c_{n}}\to \frac{1}{\E[\eta]} = \frac1{m^\eta} \quad \text{as } n\to\infty\quad\text{a.\,s.}
\end{equation}
Note that $\tau_1(n) \leq \tau_2(n)$ and
\begin{equation*}
\{\tau_1(n)=\tau_2(n)\}=\{T_{\tau_1(n)-1}+\eta_{\tau_1(n)}=c_n\}\cup\{T_{\tau_1(n)-1}+\eta_{\tau_1(n)}>c_n,J_{\tau_1(n)}=1\}.
\end{equation*}
We start by proving that $\tau_1(n)$ and $\tau_2(n)$ are uniformly close on the $\sqrt{c_n}$-scale.

\begin{lemma}	\label{Lem:tau_2-tau_1->0 uniformly in probab}
For arbitrary $T>0$ we have
\begin{equation*}
\sup_{t\in [0,\,T]} \frac{|\tau_2(nt)-\tau_1(nt)|}{\sqrt{c_n}}\Probto 0\quad \text{as } n \to \infty.
\end{equation*}
\end{lemma}

\begin{proof}
By the regular variation of $(c_n)_{n \in \N}$ it is enough to prove the claim for $T=1$. 
We have the following bound
\begin{align}\label{eq:tau_difference_proof0}
\tau_2(n)-\tau_1(n)
&=\inf\Bigg\{k\geq 0 : T_{\tau_1(n)}+\sum_{j=1}^{k}Y_{\tau_{1}(n)+j}=c_n\Bigg\}	\notag	\\
&\leq \inf\Bigg\{k\geq 0 : T_{\tau_1(n)}+\sum_{j=1}^{k}\eta_{\tau_{1}(n)+j}\1_{\{J_{\tau_{1}(n)+j}=1\}}\geq c_n\Bigg\},
\end{align}
and therefore 
\begin{multline}	\label{eq:tau_difference_proof1}
\sup_{t\in [0,1]}|\tau_2(nt)-\tau_1(nt)|	\\
\leq \inf\Bigg\{k\geq 0 : \inf_{t\in [0,1] }\sum_{j=1}^{k}\eta_{\tau_{1}(\lfloor nt \rfloor)+j}\1_{\{J_{\tau_{1}(\lfloor nt \rfloor)+j}=1\}}\geq \sup_{t\in [n^{-1},1]}\left(c_{\lfloor nt \rfloor}-T_{\tau_1(nt)}\right)\Bigg\}.
\end{multline}
Furthermore, $c_m-T_{\tau_1(m)}\leq \eta_{\tau_1(m)}$ and thus
\begin{equation*}
\sup_{t\in [n^{-1},1]}\left(c_{\lfloor nt\rfloor}-T_{\tau_1(nt)}\right)\leq \max_{t\in [n^{-1},1]}\eta_{\tau_1(nt)}\leq \max_{m=1,\ldots,\tau_1(n)}\eta_m,
\end{equation*}
which in view of \eqref{eq:tau_difference_proof1} yields 
\begin{equation}	\label{eq:tau_difference_proof11}
\sup_{t\in [0,1]}|\tau_2(nt)-\tau_1(nt)|
\leq \inf\Bigg\{k\geq 0 : \min_{m=0,\ldots,\tau_1(n)}\sum_{j=1}^{k}\eta_{m+j}\1_{\{J_{m+j}=1\}}
\geq \max_{m=1,\ldots,\tau_1( n)}\eta_m\Bigg\}.
\end{equation}
From \eqref{eq:SLLN tau_1(n)} we conclude that
for any $a>\tfrac{1}{m^\eta}$ there exists a random $n_0\in\N$ such that for all $n \geq n_0$, we have
\begin{equation*}
\sup_{t\in [n^{-1},1]}\left(c_{\lfloor nt\rfloor}-T_{\tau_1(nt)}\right)
\leq \max_{k=1,\ldots,\lfloor a c_{n}\rfloor }\eta_k.
\end{equation*}
This further implies
\begin{equation}\label{eq:tau_difference_proof13}
\sup_{t\in [0,1]}|\tau_2(nt)-\tau_1(nt)|
\leq \inf\bigg\{k\geq 0 : \min_{m=0,\ldots,\lfloor a c_n\rfloor}\sum_{j=1}^{k}\eta_{m+j}\1_{\{J_{m+j}=1\}}\geq \max_{m=0,\ldots,\lfloor a c_n\rfloor}\eta_m\bigg\}
\end{equation}
if $n \geq n_0$.
Using the assumption $\E[\eta^2]<\infty$ it is not difficult to check that
\begin{equation}	\label{eq:tau_difference_proof2}
c_n^{-\frac12} \max_{k=1,\ldots,\lfloor a c_n\rfloor }\eta_k \Probto 0	\quad \text{as } n \to \infty.
\end{equation}
Define
\begin{equation*}
\widetilde{S}_0 \defeq 0
\quad	\text{and}	\quad
\widetilde{S}_n=\sum_{k=1}^{n}\eta_k\1_{\{J_k=1\}},\ n \in \N
\end{equation*}
and note that by \eqref{eq:tau_difference_proof13} and \eqref{eq:tau_difference_proof2}
it is enough to prove that for arbitrary $\varepsilon>0$ there exist $\delta>0$ such
\begin{equation}\label{eq:tau_difference_proof3}
\Prob\bigg(\min_{m=1,\ldots,\lfloor a c_n\rfloor}
(\widetilde{S}_{m+\lfloor \varepsilon  \sqrt{c_n}\rfloor}-\widetilde{S}_m)\geq \delta  \sqrt{c_n}\bigg) \to 1
\quad	\text{as } n \to \infty.
\end{equation}
Fix $\varepsilon>0$ and let us show that \eqref{eq:tau_difference_proof3}
holds for any $\delta\in (0,\varepsilon \E \widetilde{S}_1)=(0,\varepsilon q m^{\eta})$. To this end, fix arbitrary such $\delta$ and write 
\begin{align*}
\Prob\Big(\min_{m=1,\ldots,\lfloor a c_n\rfloor}(\widetilde{S}_{m+\lfloor \varepsilon  \sqrt{c_n}\rfloor}-\widetilde{S}_m)
< \delta  \sqrt{c_n}\Big)
&\leq \sum_{m=1}^{\lfloor a c_n\rfloor}\Prob\Big(\widetilde{S}_{m+\lfloor \varepsilon  \sqrt{c_n}\rfloor}-\widetilde{S}_m
< \delta  \sqrt{c_n}\Big)\\
&=\lfloor a c_n\rfloor \Prob\Big(\widetilde{S}_{\lfloor \varepsilon  \sqrt{c_n}\rfloor} < \delta  \sqrt{c_n}\Big).
\end{align*}
For every $\lambda>0$ we have by Markov's inequality
\begin{equation*}
\Prob\Big(\widetilde{S}_{\lfloor \varepsilon  \sqrt{c_n}\rfloor} < \delta  \sqrt{c_n}\Big)
=\Prob\Big(e^{-\lambda \widetilde{S}_{\lfloor \varepsilon  \sqrt{c_n}\rfloor}} > e^{-\lambda\delta  \sqrt{c_n}}\Big)
\leq e^{\lambda\delta  \sqrt{c_n}}(\E[e^{-\lambda \eta \1_{\{J=1\}}}])^{\lfloor \varepsilon  \sqrt{c_n}\rfloor}.
\end{equation*}
It remains to note that $e^{\lambda\delta\varepsilon^{-1}}(\E[e^{-\lambda \eta \1_{\{J=1\}}}])<1$
for $\delta\in(0,\varepsilon q m^{\eta})$ and sufficiently small $\lambda > 0$. The proof is complete.
\end{proof}

For the proof of Theorem \ref{Thm:SLLN L_n's and N_n's} we also need a counterpart of the above lemma
for convergence in the almost sure sense.

\begin{lemma}\label{lem:tau_differecne_as}
It holds that
\begin{equation*}
\lim_{n\to\infty}\frac{\tau_2(n)-\tau_1(n)}{c_n}= 0	\quad	\text{a.\,s.}
\end{equation*}
\end{lemma}
\begin{proof}
In view of \eqref{eq:tau_difference_proof0} for every $\varepsilon>0$ it holds
\begin{align}	
\{\tau_2(n)-\tau_1(n) > \lfloor \varepsilon c_n\rfloor\}
&\subseteq
\Bigg\{\sum_{j=1}^{\lfloor \varepsilon c_n\rfloor}\eta_{\tau_{1}(n)+j}\1_{\{J_{\tau_{1}(n)+j}=1\}}< c_n-T_{\tau_1(n)}\Bigg\}	\notag	\\
&=\Big\{\widetilde{S}_{\tau_1(n)+\lfloor \varepsilon c_n\rfloor}-\widetilde{S}_{\tau_1(n)}< c_n-T_{\tau_1(n)}\Big\}.\label{eq:tau_difference_proof_as1} 
\end{align}
By the classical strong law of large numbers for $(\widetilde{S}_k)_{k \in \N_0}$ and in view of \eqref{eq:SLLN tau_1(n)}
\begin{equation*}
\lim_{n\to\infty}\frac{\widetilde{S}_{\tau_1(n)+\lfloor \varepsilon c_n\rfloor}-\widetilde{S}_{\tau_1(n)}}{c_n}
= \varepsilon q m^{\eta} > 0
\quad\text{and}\quad
\lim_{n\to\infty}\frac{T_{\tau_1(n)}}{c_n}=1	\quad	\text{a.\,s.}
\end{equation*}
Therefore, with probability one, the event in \eqref{eq:tau_difference_proof_as1} occurs only for finitely many $n$.
\end{proof}

\begin{proof}[Proof of Theorem \ref{Thm:SLLN L_n's and N_n's}]
Note that for $i=0,1$ we can write
\begin{align}
L_{n}^{(i)}=\sum_{k=1}^{n}\1_{\{X_k+Y_k>0,I_k=i\}}
&=\sum_{k=1}^{(\tau_1(n)-1)\wedge n} \!\!\!\!  \1_{\{\xi_k+\eta_k>0,I_k=i\}}	\notag	\\
&\hphantom{=\ }
+\sum_{k=\tau_1(n)}^{\tau_2(n)\wedge n} \!\! \1_{\{X_k+Y_k>0,I_k=i\} }
+\sum_{k=\tau_2(n)+1}^{n} \!\!\!  \1_{\{\theta_k>0,I_k=i\}}.		\label{eq:decompostion of L_n^i}
\end{align}
From this representation all the claims follow
immediately from the classical strong law of large numbers and the fact
\begin{equation*}
\lim_{n\to\infty}\frac{\tau_1(n)}{c_n}=\lim_{n\to\infty}\frac{\tau_2(n)}{c_n}=\frac{1}{m^{\eta}}\quad \text{a.\,s.},
\end{equation*}
which is a consequence of \eqref{eq:SLLN tau_1(n)} and Lemma \ref{lem:tau_differecne_as}.
\end{proof}

\subsection{Proof of Theorem \ref{Thm:joint CLT}}

Define $X_n\in D([0,\infty),\R^5)$ and $Y_n\in D([0,\infty),\R^2)$ via
\begin{align*}
X_n(t)
&= \sum_{k=1}^{\lfloor nt \rfloor}\left(\1_{\{\xi_k+\eta_k>0,I_k=0\}},\1_{\{\xi_k+\eta_k>0,I_k=1\}} ,I_k,\xi_k,\eta_k\right)\\
\text{and}	\quad
Y_n(t)
&= \sum_{k=1 }^{\lfloor nt \rfloor}\left(\1_{\{\theta_k>0,I_k=0\}},\1_{\{\theta_k>0,I_k=1\}}\right)
\end{align*}
for $t \geq 0$.
The following proposition is the key ingredient in the proof of Theorem \ref{Thm:joint CLT}.

\begin{proposition}\label{prop:joint_convergence_direct}
If the assumptions of Theorem \ref{Thm:joint CLT} hold, then, as $n\to\infty$,
\begin{equation*}
\bigg(\frac{X_n(t)-nt\big((1\!-\!p)p_0,pp_1,p,m^{\xi},m^{\eta}\big)}{\sqrt{n}},
\frac{Y_n(t)-nt\big((1\!-\!p)p_\theta,pp_\theta\big)}{\sqrt{n}}\bigg)_{\! t\geq 0}
\! \Longrightarrow ({\bf B}(t))_{t\geq 0},
\end{equation*}
where $({\bf B}(t))_{t\geq 0}$ is a centered Brownian motion with covariance matrix $\mathbf{V}$ as in \eqref{eq:covariance matrix V}.
\end{proposition}
\begin{proof}
The convergence follows from Donsker's invariance principle
since $(X_n(t),Y_n(t))$ is the sum of independent identically distributed random vectors in $\R^7$ with finite second moments.
The increment vectors have expectation
\begin{multline*}
\E\big[\big(\1_{\{\xi_k+\eta_k>0,I_k=0\}},\1_{\{\xi_k+\eta_k>0,I_k=1\}},
I_k,\xi_k,\eta_k,\1_{\{\theta_k>0,I_k=0\}},\1_{\{\theta_k>0,I_k=1\}}\big)\big]	\\
= \big(p_0(1-p),p_1p,p,m^\xi,m^\eta,p_\theta(1-p),p_\theta p\big).
\end{multline*}
The explicit form of the covariance matrix now results from elementary, yet cumbersome calculations.
For example the entry at the first row and fourth column of ${\bf V}$ can be calculated as follows:
\begin{align*}
\Cov\big[\1_{\{\xi_k+\eta_k>0,I_k=0\}},\xi_k\big]
&=\E\big[\xi_k\1_{\{\xi_k+\eta_k>0,I_k=0\}}\big]-m^{\xi} \Prob(\xi_k+\eta_k>0,I_k=0)	\\
&=\E \big[\xi_k\1_{\{\xi_k>0,I_k=0\}}\big] - m^{\xi} (1\!-\!p) p_0	\\
&=\E \big[\xi_k\1_{\{I_k=0\}}\big] - m^{\xi}(1\!-\!p)p_0	\\
&=(1\!-\!p)(m_0^{\xi}- m^{\xi}p_0).
\end{align*}
\end{proof}

In what follows, we shall frequently use the following two facts (see the Lemma on p.\;151 in \cite{Billingsley:1999}
and \cite[Theorem 3.1]{Whitt:1980}):
\begin{itemize}
	\item[Fact 1:]
		the addition mapping $+:D([0,\infty),\R)\times D([0,\infty),\R)\mapsto D([0,\infty))$ defined by $(f+g)(x)=f(x)+g(x)$,
		is continuous with respect to the $J_1$-topology at all points $(f,g)$ such that both $f$ and $g$ are continuous;
	\item[Fact 2:]
		the composition mapping $\circ:D([0,\infty),\R)\times D([0,\infty),\R)\mapsto D([0,\infty))$ defined by $(f\circ g)(x)=f(g(x))$
		is continuous with respect to the $J_1$-topology at all points $(f,g)$ such that both $f$ and $g$ are continuous
		and $g$ is nondecreasing.
\end{itemize}

\begin{proof}[Proof of Theorem \ref{Thm:joint CLT}]
From \cite[Corollary 7.3.1]{Whitt:2002}, \eqref{eq:c_n reg var} and Fact 2, we infer
\begin{equation}\label{eq:tau_1_convergence}
\bigg(\frac{\tau_1(nt)-\tfrac{1}{m^\eta}c_{\lfloor nt\rfloor}}{(m^{\eta})^{-3/2}\sqrt{c_n}}\bigg)_{\! t\geq 0}
\!\Longrightarrow (-B_5(t^{\rho}))_{t\geq 0},
\end{equation}
and, in view of Lemma \ref{Lem:tau_2-tau_1->0 uniformly in probab}, the same relation for $\tau_2(nt)$.
By applying Corollary 13.8.1 in \cite{Whitt:2002} 
and again Lemma \ref{Lem:tau_2-tau_1->0 uniformly in probab},
we can further extend the convergence in Proposition \ref{prop:joint_convergence_direct}
to a joint convergence:
\begin{multline}\label{eq:joint _convergence_with_nu1}
\left(\frac{X_n(t)-nt\left(p_0(1-p),pp_1,p,m^{\xi},m^{\eta}\right)}{\sqrt{n}},\frac{Y_n(t)-nt\left((1-p)p_\theta,pp_\theta\right)}{\sqrt{n}},\right.\\
\left.\frac{\tau_1(nt)-\tfrac{1}{m^\eta}c_{\lfloor nt\rfloor}}{(m^{\eta})^{-3/2}\sqrt{c_n}},\frac{\tau_2(nt)-\tfrac{1}{m^\eta}c_{\lfloor nt\rfloor}}{(m^{\eta})^{-3/2}\sqrt{c_n}}\right)_{t\geq 0}
\Longrightarrow ({\bf B}(t),-B_5(t^{\rho}),-B_5(t^{\rho}))_{t\geq 0}.
\end{multline}
As in \eqref{eq:decompostion of L_n^i}, for $i=0,1$, we can write
\begin{align*}	
L_{\lfloor nt\rfloor}^{(i)}
= \sum_{k=1}^{\lfloor nt\rfloor} \1_{\{X_k+Y_k>0,I_k=i\} }
&=\sum_{k=1}^{(\tau_1(nt)-1)\wedge \lfloor nt\rfloor} \!\!\!\!\! \1_{\{\xi_k+\eta_k>0,I_k=i\} }\\
&\hphantom{=~}+\sum_{k=\tau_1(nt)\wedge (\lfloor nt\rfloor+1)}^{\tau_2(nt)\wedge \lfloor nt\rfloor} \!\!\!\!\!\!  \1_{\{X_k+Y_k>0,I_k=i\} }
+\sum_{k=\tau_2(nt)\wedge \lfloor nt\rfloor+1}^{\lfloor nt\rfloor} \!\!\!\!\! \1_{\{\theta_k>0,I_k=i\}}.
\end{align*}
The second summand above is bounded by $\tau_2(nt)-\tau_1(nt)+1$
and thus the supremum over $t$ in a compact interval divided by $\sqrt{c_n}=O(\sqrt{n})$
converges to zero in probability by Lemma \ref{Lem:tau_2-tau_1->0 uniformly in probab} as $n \to \infty$.
The behavior of the second and the third summand strongly depends on whether
$\lim_{n \to \infty} \frac{c_n}{n}$ is smaller, larger or equal to $\tfrac{1}{m^\eta}$.

Given $0<a<b$ and $i=1,2$ we put
\begin{equation*}
A^{a,b,i}_n\defeq\{\tau_i(nt)\leq \lfloor nt\rfloor	\text{ for all } t\in[a,b]\}
\quad	\text{and}	\quad
B^{a,b,i}_n\defeq\{\tau_i(nt)\geq \lfloor nt\rfloor+1	\text{ for all }t\in[a,b]\}.
\end{equation*}

\noindent
We first deal with the case $c_{\infty}\in (0,\tfrac{1}{m^\eta})$.
In this case, we necessarily have $\rho=1$.
By Eq.~\eqref{eq:tau_1_convergence}, Lemma \ref{Lem:tau_2-tau_1->0 uniformly in probab}
and the uniform convergence theorem for regularly varying functions we obtain for $i=1,2$ and arbitrary $0<a<b$
\begin{equation}\label{eq:uniform WLLN for tau_1}
\sup_{t\in [a,b]}\left|\frac{\tau_i(nt)}{c_n}-\frac{t}{m^{\eta}}\right|\Probto 0\quad \text{as } n \to \infty,
\end{equation}
and therefore
\begin{equation}	\label{eq:tau_1(nt) leq nt unif}
\lim_{n\to\infty }\Prob(A^{a,b,1}_n\cap A^{a,b,2}_n)=1.
\end{equation}
For $i=1,2$, put 
\begin{equation}\label{eq:l_n_i_hat_decompostion}
\widehat{L}_{\lfloor nt\rfloor}^{(i)}=\sum_{k=1}^{\tau_1(nt)-1}\1_{\{\xi_k+\eta_k>0,I_k=i\} }+\sum_{k=\tau_2(nt)}^{\lfloor nt\rfloor}\1_{\{\theta_k>0,I_k=i\}}.
\end{equation}
From what have proved above it is clear that it suffices to check \eqref{eq:joint FCLT small c_n}
with $L_{\lfloor nt\rfloor}^{(i)}$ replaced by $\widehat{L}_{\lfloor nt\rfloor}^{(i)}$, $i=1,2$.
For typographical reasons we shall write convergences of various components in separate formulas
keeping in mind that they actually converge jointly in view of \eqref{eq:joint _convergence_with_nu1}.
Firstly,
\begin{multline*}
\bigg(\frac{\sum_{k=1}^{\lfloor c_n t \rfloor-1}\1_{\{\xi_k+\eta_k>0,I_k=0\}}-c_n t p_0(1-p)}{\sqrt{c_n}},\frac{\sum_{k=1}^{\lfloor c_n t \rfloor-1}\1_{\{\xi_k+\eta_k>0,I_k=1\}}-c_n t p p_1)}{\sqrt{c_n}}\bigg)_{\!t \geq 0}	\\
\Longrightarrow (B_1(t),B_2(t))_{t\geq 0}	\quad	\text{as } n\to\infty,
\end{multline*}
and therefore using Fact 2 (continuity of the composition mapping),
the continuous mapping theorem and \eqref{eq:uniform WLLN for tau_1}
\begin{multline*}
\bigg(\frac{\sum_{k=1}^{\tau_1(nt)-1}\1_{\{\xi_k+\eta_k>0,I_k=0\}} - \tau_1(nt) p_0(1\!-\!p)}{\sqrt{c_n}},
\frac{\sum_{k=1}^{\tau_1(nt)-1}\1_{\{\xi_k+\eta_k>0,I_k=1\}}- \tau_1(nt) p p_1)}{\sqrt{c_n}}\bigg)_{\! t>0}	\\
\Longrightarrow \Big(B_1\left(\tfrac{t}{m^{\eta}}\right),B_2\left(\tfrac{t}{m^{\eta}}\right)\Big)_{\! t>0}
\quad \text{ as } n\to\infty.
\end{multline*}
Secondly, using Fact 1 (continuity of addition)
and convergence of the last components in \eqref{eq:joint _convergence_with_nu1}
we deduce
\begin{multline}	\label{eq:case1_key_convergence_1}
\bigg(\frac{\sum_{k=1}^{\tau_1(nt)-1}\1_{\{\xi_k+\eta_k>0,I_k=0\}} - p_0(1\!-\!p)\tfrac{1}{m^\eta}c_{\lfloor nt\rfloor}}{\sqrt{c_n}},\frac{\sum_{k=1}^{\tau_1(nt)-1}\1_{\{\xi_k+\eta_k>0,I_k=1\}} -  p p_1\tfrac{1}{m^\eta}c_{\lfloor nt\rfloor})}{\sqrt{c_n}}\bigg)_{t> 0}\\
\Longrightarrow \Big(B_1\left(\tfrac{t}{m^{\eta}}\right)-p_0(1\!-\!p)(m^{\eta})^{3/2}B_5(t),B_2\left(\tfrac{t}{m^{\eta}}\right)-pp_1(m^{\eta})^{3/2}B_5(t)\Big)_{t > 0} \quad	\text{as } n\to\infty.
\end{multline}
In the same vein,
\begin{multline}\label{eq:case1_key_convergence_2}
\bigg(\frac{\sum_{k=1}^{\tau_2(nt)-1}\1_{\{\theta_k>0,I_k=0\}} - p_\theta(1\!-\!p)\tfrac{1}{m^\eta}c_{\lfloor nt\rfloor}}{\sqrt{c_n}},\frac{\sum_{k=1}^{\tau_2(nt)-1}\1_{\{\theta_k>0,I_k=1\}} -  p p_\theta\tfrac{1}{m^\eta}c_{\lfloor nt\rfloor})}{\sqrt{c_n}}\bigg)_{t> 0}\\
\Longrightarrow \Big(B_6\left(\tfrac{t}{m^{\eta}}\right)-p_\theta(1\!-\!p)(m^{\eta})^{3/2}B_5(t),B_7\left(\tfrac{t}{m^{\eta}}\right)-pp_\theta(m^{\eta})^{3/2}B_5(t)\Big)_{t > 0}\quad \text{as } n \to \infty
\end{multline}
Replacing $\sqrt{c_n}$ in the denominators by $\sqrt{c_{\infty}n}$ and summing everything up we get
\begin{align}
\bigg(&\frac{L_{\lfloor nt\rfloor}^{(0)}-(1\!-\!p)\tfrac{1}{m^\eta}(p_0-p_\theta) c_{\lfloor nt\rfloor}-nt(1\!-\!p)p_\theta}{\sqrt{n}},\frac{L_{\lfloor nt\rfloor}^{(1)}	- p\tfrac{1}{m^\eta}(p_1-p_\theta)c_{\lfloor nt\rfloor}-nt pp_\theta}{\sqrt{n}}\bigg)_{t> 0}	\notag	\\
&\Longrightarrow
\Big(\sqrt{c_{\infty}}\left(B_1\left(\tfrac{t}{m^{\eta}}\right)-p_0(1\!-\!p)(m^{\eta})^{3/2}B_5(t)-B_6\left(\tfrac{t}{m^{\eta}}\right)+p_\theta(1\!-\!p)(m^{\eta})^{3/2}B_5(t)\right)+B_6(t),	\notag	\\
&\hphantom{\Longrightarrow\Big(\ }
\sqrt{c_{\infty}}\left(B_2\left(\tfrac{t}{m^{\eta}}\right)-pp_1(m^{\eta})^{3/2}B_5(t)-B_7\left(\tfrac{t}{m^{\eta}}\right)+pp_\theta(m^{\eta})^{3/2}B_5(t))\Big)+B_7(t)\right)_{t > 0}\label{eq:case1_key_convergence_3}
\end{align}
as $n \to \infty$.
It remains to note that \eqref{eq:case1_key_convergence_3} holds jointly with
\begin{equation*}
\Bigg(\frac{N^{(0)}_{\lfloor nt\rfloor}-(1\!-\!p)nt}{\sqrt{n}},\frac{N^{(1)}_{\lfloor nt\rfloor}-pnt}{\sqrt{n}}\Bigg)_{t \geq 0}
\Longrightarrow (-B_3(t),B_3(t))_{t \geq 0},
\end{equation*}
and together this is \eqref{eq:joint FCLT small c_n}.

\vspace{0.5cm}
\noindent
{\sc Case $c_{\infty}=0$.} In this case \eqref{eq:tau_1(nt) leq nt unif} still holds but there are significant simplifications. First of all note that in this case we can have $\rho\leq 1$ and thus in \eqref{eq:uniform WLLN for tau_1} $\tfrac{1}{m^\eta}t$ must be replaced by $\tfrac{1}{m^\eta}t^{\rho}$. Further, in \eqref{eq:case1_key_convergence_1} and \eqref{eq:case1_key_convergence_2} upon replacing $\sqrt{c_n}$ in the denominator by $\sqrt{n}$ the limit becomes identical zero and, thus we have the same convergence \eqref{eq:case1_key_convergence_3} but with $c_{\infty}=0$ on the right-hand side.	\smallskip

\noindent
Finally, we deal with the case $c_{\infty}>\tfrac{1}{m^\eta}$.
In this case \eqref{eq:uniform WLLN for tau_1} implies
\begin{equation*}
\lim_{n\to\infty}\Prob(B^{a,b,1}_{n}\cap B^{a,b,1}_{n})=1.
\end{equation*}
Similar to \eqref{eq:l_n_i_hat_decompostion}, for $i=1,2$, we now put 
\begin{equation*}
\widehat{L}^{(i)}_{\lfloor nt \rfloor}\defeq\sum_{k=1}^{\lfloor nt \rfloor}\1_{\{\xi_k+\eta_k>0,I_k=i\}},
\end{equation*}
and note that now $L^{(i)}_{\lfloor nt \rfloor}$ can be replaced by  $\widehat{L}^{(i)}_{\lfloor nt \rfloor}$ in \eqref{eq:joint FCLT large c_n}.
After this replacement \eqref{eq:joint FCLT large c_n} is just a part of \eqref{eq:joint _convergence_with_nu1}.
\end{proof}

We now turn to the proof of Theorem \ref{Thm:cvgce chi^2 statistics}.
As a first step, we notice that Corollary \ref{Cor:cvgce chi^2 statistics}
follows from Theorem \ref{Thm:joint cvgce var in chi^2 statistics} and the continuous mapping theorem.
Theorem \ref{Thm:joint cvgce var in chi^2 statistics}, in turn, follows immediately
from Theorem \ref{Thm:joint CLT}. It thus remains to deduce Theorem \ref{Thm:cvgce chi^2 statistics}
from Theorem \ref{Thm:joint cvgce var in chi^2 statistics}.

\begin{proof}[Proof of Theorem \ref{Thm:cvgce chi^2 statistics}]
Our aim is to show how to calculate the distribution of the variable 
\begin{align*}
\chi_{\infty}^2& \defeq \frac{(d_2+G_2-p_\theta G_1-(1\!-\!p)(d_2\!+\!d_3\!+\!G_2\!+\!G_3))^2}{(1-p_\theta)p_\theta(1-p)}
+\frac{(d_3+G_3+p_\theta G_1-p(d_2\!+\!d_3\!+\!G_2\!+\!G_3))^2}{(1-p_\theta)p_\theta p}\\
&=\frac{(pd_2 -(1\!-\!p)d_3-p_\theta G_1+p G_2 - (1\!-\!p)G_3)^2}{(1-p_\theta)p_\theta(1-p)}
+\frac{((1\!-\!p)d_3-pd_2+p_\theta G_1-pG_2+(1\!-\!p)G_3)^2}{(1-p_\theta)p_\theta p},
\end{align*}
where ${\bf G}\defeq(G_1,G_2,G_3) $ is a centered Gaussian vector with covariance matrix ${\bf V_1}$. A simple calculation shows that
\begin{equation*}
\chi_{\infty}^2=\frac{(p_\theta G_1-pG_2+(1-p)G_3-(pd_2-(1-p)d_3))^2}{(1-p_\theta)p_\theta (1-p)p}.
\end{equation*}
Note that $p_\theta G_1-pG_2+(1-p)G_3$ has a centered normal distribution with the variance
\begin{align*}
\Var[p_\theta G_1-pG_2+(1-p)G_3]
&=p_{\theta}^2\Var[G_1]+p^2\Var[G_2] + (1-p)^2\Var[G_3]-2 p p_{\theta}\Cov[G_1,G_2]\\
&\hphantom{=} +2p_{\theta}(1-p)\Cov[G_1,G_3]-2p(1-p)\Cov[G_2,G_3]\\
&=p_\theta (1-p_\theta)p(1-p).
\end{align*}
Thus, for $\Normal$ having the standard normal distribution, we see that
\begin{equation*}
\chi_{\infty}^2\disteq \left(\Normal-\frac{pd_2-(1-p)d_3}{\sqrt{p p_\theta (1-p)(1-p_\theta)}}\right)^2=\left(\Normal-\frac{d_\infty(p_0-p_1)\sqrt{p(1-p)}}{m^{\eta}\sqrt{p_\theta(1-p_\theta)}}\right)^2.
\end{equation*}
The proof is complete.
\end{proof}

\section{Conclusions}

Starting from the observation that the standard for testing product changes on e-commerce websites is large scale hypothesis testing with statistical tests based on the assumption of independent samples such as the chi-squared test, we have suggested a new model for the samples which incorporates shared inventories.
This model introduces new dependencies.
Our main result is the calculation of the asymptotic law of the chi-squared test statistics under the new model assumptions in the critical regime where the number of items of a popular good is of the order of the square root of the sample size. Website versions that greedily sell the popular good have an initial advantage
in the number of sales of the order of the square root of the sample size, which is the order of the overall random fluctuations. Thus the initial advantage has an impact on the probability of rejecting the hypothesis.
We have demonstrated in examples that this may lead to both, arbitrarily high false-positive as well as arbitrarily high false-negative rates. This questions the assumption implicit in the industry standard that dependencies are small enough to be ignored. Moreover, it suggests that the present standard of A/B testing favors algorithms
that are designed to be good in competition against others, but not necessarily good when used on separate inventory.

Our work may be extended in the future in several directions.
On the one hand, our results may be used to construct tests for the model that keep the significance level.
On the other hand, the model may be extended to incorporate more features of real samples
such as a priori information about website visitors etc.

\section{Acknowledgements}

A.\,M.\ was supported by the Ulam programme funded by the Polish national agency for academic exchange (NAWA), project no. PPN/ULM/2019/1/00004/DEC/1. 
The authors would like to thank Tanja Matic and Onno Zoeter for many helpful discussions. We are particularly grateful to Onno Zoeter for communicating the essence of Example \ref{Exa:picky customers}.


\bibliographystyle{plain}
\bibliography{Test.bib} 


\end{document}